\def\publicationversion{1}
\documentclass[10pt,english,a4paper]{amsart}
\usepackage{iftex}
\ifPDFTeX\pdfoutput=1\fi

\usepackage[utf8]{inputenc}
\usepackage[T1]{fontenc}
\usepackage{lmodern}
\usepackage{babel}
\usepackage{amsmath,amssymb,amsthm}
\usepackage[marginratio={1:1,1:1},totalwidth=400pt,totalheight=600pt]{geometry}
\usepackage{microtype}
\usepackage{hyperref}
\usepackage{url}
\usepackage{longtable,booktabs,array}
\usepackage{needspace}
\usepackage{xcolor}
\hypersetup{hidelinks}
\theoremstyle{plain}
\newtheorem{theorem}{Theorem}[section]
\newtheorem{lemma}[theorem]{Lemma}
\newtheorem{proposition}[theorem]{Proposition}
\newtheorem{corollary}[theorem]{Corollary}
\newtheorem*{mainthm}{Theorem A}

\theoremstyle{definition}
\newtheorem{definition}[theorem]{Definition}
\newtheorem{example}[theorem]{Example}
\newtheorem{question}[theorem]{Question}

\theoremstyle{remark}
\newtheorem{remark}[theorem]{Remark}

\newcommand{\Z}{\mathbb{Z}}
\newcommand{\F}{\mathbb{F}}
\newcommand{\bzero}{\mathbf{0}}
\newcommand{\bone}{\mathbf{1}}
\newcommand{\GL}{\operatorname{GL}}
\newcommand{\AGL}{\operatorname{AGL}}
\newcommand{\intset}{\operatorname{int}}

\DeclareMathOperator{\ord}{ord}

\newcommand{\PSedited}[1]{\ifdefined\publicationversion#1\else{\color{red}#1}\fi}

\begin{document}
	
\title[Equivalences of rotation-symmetric S-boxes]
{Affine and cyclic affine equivalence of rotation-symmetric S-boxes}
\author[Omland]{Tron Omland}
\address{Norwegian National Security Authority (NSM) and University of Oslo, Norway}
\email{tron.omland@gmail.com}
\author[St\u{a}nic\u{a}]{Pantelimon St\u{a}nic\u{a}}
\address{Department of Applied Mathematics, Naval Postgraduate School,
Monterey, CA 93943-5212, U.S.A.}
\email{pstanica@nps.edu}
\date{November 1, 2026}

\relax
\relax\relax


\begin{abstract}
A rotation-symmetric S-box on $n$ bits is a vectorial Boolean function that
commutes with the right cyclic shift and is therefore determined by a single
Boolean function. The map $\chi$ used in Keccak is a well-known example.
Ordinary affine equivalence allows arbitrary invertible affine
transformations, which need not preserve rotation symmetry.
We introduce \emph{cyclic affine equivalence}, in which the
implementing transformations each have an invertible $k$-circulant linear part with
$\gcd(k,n)=1$ and a translation vector fixed by the shift.
For $n\ge3$, the two relations give the same classes of rotation-symmetric permutations
exactly when $n$ is an odd prime for which $2$ is a primitive root.
In every other dimension $n\ge3$, explicit constructions exhibit affine
classes containing several cyclic affine classes. We derive exact class-size
formulas, count cyclic affine classes for $3\le n\le9$, and classify a
four-bit family of $\chi$-siblings under both relations. Finally, we
characterize when a permutation is affinely equivalent to a rotation-symmetric
one and apply the criterion to the Ascon S-box.
\end{abstract}

\maketitle
\section{Introduction}\label{sec:intro}

One of the most basic primitives in symmetric cryptography is an S-box, or a
``substitution box'', which, mathematically, is a map on the set of $n$-bit
vectors. Symmetric ciphers are often built from S-boxes and a few other
operations that are usually linear. Substitution-permutation networks are
of this type, including the block cipher standard AES. As the main nonlinear
part, S-boxes play a central role in providing confusion, and the security
of a cipher often relies heavily on the cryptographic properties of the
S-box involved.

Since lookup tables tend to have a large implementation cost, using an S-box
with an easy description is favorable. We therefore consider rotation-symmetric
S-boxes.
Let $\F_2^n$ denote the vector space of $n$-bit vectors and let $S$ be the
right cyclic shift,
$S(x_1,\ldots,x_n)=(x_n,x_1,\ldots,x_{n-1})$.
A vectorial Boolean function $F\colon\F_2^n\to\F_2^n$ is
\emph{rotation-symmetric}, or \emph{shift-invariant}, if $FS=SF$.
It is determined by a single Boolean function $f$ through
\[
\begin{aligned}
 F(x_1,\ldots,x_n)={}&\bigl(f(x_1,\ldots,x_n),f(x_2,\ldots,x_n,x_1),\ldots,\\
 &\hspace{35mm}f(x_n,x_1,\ldots,x_{n-1})\bigr).
\end{aligned}
\]
Rotation-symmetric permutations with suitable cryptographic properties are
therefore candidates for use in symmetric ciphers, including lightweight
cryptography. Our focus is on when two such permutations should be regarded
as equivalent. We also consider the more general \emph{cyclic} maps,
which satisfy $FS=S^kF$ for some $k$ coprime to $n$, called their
\emph{shift parameter}.

Two S-boxes $F$ and $G$ are \emph{affinely equivalent} if $F\alpha=\beta G$
for invertible affine transformations $\alpha$ and $\beta$.
This relation preserves algebraic degree, the differential spectrum, and the
absolute Walsh spectrum. Classification up to affine equivalence avoids listing multiple
affine variants with the same invariants. However, arbitrary affine
transformations need not preserve rotation symmetry.
We introduce \emph{cyclic affine equivalence} by requiring each transformation
to have an invertible $k$-circulant linear part with $\gcd(k,n)=1$ and a
translation vector fixed by $S$. Equivalently, each transformation $\gamma$
satisfies $\gamma S\gamma^{-1}=S^\ell$ for some $\ell$ coprime to $n$.
Section~\ref{sec:equiv} characterizes the relation through the generating
Boolean functions. As a restriction of affine equivalence, it preserves
the same invariants.

Do these two relations give the same classification? Our main result gives
the exact condition on the dimension.

\relax
\begin{mainthm}
\relax\relax
Let $n\ge3$. Affine and cyclic affine equivalence give the same classification
of cyclic permutations of $\F_2^n$ if and only if $n$ is an odd prime for
which $2$ is a primitive root.
\end{mainthm}

Write $\GL_n(\F_2)$ and $\AGL_n(\F_2)$ for the groups of invertible
linear and affine transformations of $\F_2^n$, respectively.
The proof uses the affine self-equivalence group
\[
 \mathcal A_F=\{\gamma\in\AGL_n(\F_2):F\gamma F^{-1}
                    \in\AGL_n(\F_2)\}.
\]
Let $F,G$ be cyclic permutations and write an affine equivalence as
$F\alpha=\beta G$. They are cyclically affine equivalent exactly when
$\langle S\rangle$ and $\alpha\langle S\rangle\alpha^{-1}$ are conjugate
in $\mathcal A_F$. For odd $n$, let $\ord_n(2)$ denote the multiplicative
order of $2$ modulo $n$, and let $\varphi$ be Euler's totient function.
If $n=p$ is prime and $\ord_p(2)=p-1$, every $p$-subgroup of
$\AGL_p(\F_2)$ is cyclic and Sylow theory supplies that conjugacy.
In every other dimension $n\ge3$, explicit constructions separate the two
relations. For odd $n$ with $\ord_n(2)<\varphi(n)$, powers of Frobenius on
one component of $\F_2[X]/(X^n-1)$ give $\ord_n(2)$ linearly conjugate maps
that are pairwise cyclically affine inequivalent. Thus one affine class
contains at least $\ord_n(2)$ cyclic affine classes. Together with the even
and square-factor constructions, this covers all remaining dimensions.
\relax

In dimension seven we also determine the exact class intersections with
a family of $49$ maps in Section~\ref{subsec:p7-family}.
\relax

Section~\ref{sec:counting} gives exact cyclic affine class-size formulas
and counts the classes of rotation-symmetric permutations for $3\le n\le9$
using Burnside's lemma. To decide cyclic affine equivalence of cyclic
permutations $F,G$, we enumerate at most $2c_n\varphi(n)$ cyclic affine
transformations $\alpha$ and test whether
$\beta=F\alpha G^{-1}$ is cyclic affine, where $c_n$ counts invertible
binary circulants.
At the primes in Theorem~A this search also decides ordinary affine
equivalence.

Section~\ref{sec:connections} applies the theory to $\chi$ and related
families. A Boolean rule $f$ on $d$ variables defines maps by
$F_n(x)_i=f(x_i,x_{i+1},\ldots,x_{i+d-1})$, with indices modulo $n$.
When $n\ge d$ and $F_n$ is bijective, we call the rule a
$(d,n)$-\emph{lifting}. Such local rules are studied in~\cite{HO1,HO2}.
The resulting permutations are
reversible cellular automata on $n$ sites with periodic boundary conditions.
The motivating example is the function
$\chi(x_1,x_2,x_3)=x_1+(x_2+1)x_3$, which gives rise to $(3,n)$-liftings
for all odd $n\geq3$~\cite{JDA-thesis}. It appears in Keccak~\cite{keccak} and in
other permutation constructions. We classify the four-bit permutation family
of $\chi$-siblings constructed in~\cite{KK26}.
Although some affine classes split in dimension four, both relations give
the same three classes on this family. We also show that cyclic affine
equivalence of maps defined by local rules on $n$ bits implies cyclic affine
equivalence on $m$ bits for every divisor $m$ of $n$, using cyclic indexing
for the rules.

Section~\ref{sec:recognition} characterizes when a permutation given as a
lookup table is affinely equivalent to a rotation-symmetric one and
reconstructs such a representative. The Ascon example recovers its known
affine relation to $\chi_5$~\cite{MMD23}. This addresses a specific symmetry
within the broader work on affine equivalence algorithms~\cite{BDBP03} and
recovery of hidden S-box structure~\cite{BPU16}.

In a normal basis of $\F_{2^n}$ over $\F_2$, squaring a field element
cyclically shifts its binary coordinates. Thus rotation-symmetric S-boxes
correspond to maps $H\colon\F_{2^n}\to\F_{2^n}$ satisfying
$H(z^2)=H(z)^2$~\cite{RBG08}. Their construction
and classification are studied in~\cite{Kavut12}. The connection between
shift invariance and linear self-equivalence is studied in~\cite{BBL21}.
Our counting results also apply to the substitutions
studied in~\cite{PM26}. Further shift-invariant families appear
in~\cite{FWYZ26,LZYH26,ZTW26}. Linear self-symmetries of inverse-based S-boxes
are studied in~\cite{ZZ26}. Scalar rotation-symmetric Boolean functions
$h$ satisfy $hS=h$. Work on these functions, such as
\cite{ChirvasituCusick20,Cusick11}, concerns this different symmetry
condition and input-only changes of variables.
\relax

The finite-ring obstruction in~\cite{Fu26} complements the lifting
catalogues in Appendix~\ref{old:tables}.
\relax

\PSedited{The preprint~\cite{OS22} introduced the foundational equivalence theory.
We prove the dimension classification and develop class enumeration
and recognition methods here.}

\relax
Appendices~\ref{old:liftings} and~\ref{old:tables} give the lifting
results and catalogues, including the definitions and examples used by
later papers.

\relax
Section~\ref{sec:open} collects open problems.
Throughout, $n\ge2$ unless stated otherwise. We write $\bzero$ and $\bone$
for the all-zero and all-one vectors. The translation vectors in cyclic affine equivalences belong to
$\{\bzero,\bone\}$.
Composition is written without the symbol $\circ$ when unambiguous.


\section{Cyclic functions and circulant matrices}\label{sec:cyclic}

The cycle of an element $x\in\F_2^n$ is the set
$c(x)=\{S^j(x):j\in\Z\}$, whose size must divide $n$.
Two cycles are either equal or disjoint, so the collection of cycles forms
a partition of $\F_2^n$.
A map is \emph{cycle-preserving} if it maps each cycle into a cycle.
It is \emph{strongly cycle-preserving} if it maps each cycle onto a cycle.

\begin{definition}
\label{def:k-si}
A function $F$ is \emph{$k$-shift-invariant} if $FS=S^kF$.
It is \emph{cyclic} if it has such a shift parameter $k$ coprime to $n$.
The case $k=1$ is rotation symmetry, or shift-invariance.
\end{definition}

Every $k$-shift-invariant map is cycle-preserving since
$F(S^ix)=S^{ki}F(x)$. A unit shift parameter makes it strongly
cycle-preserving. Its output cycle can still be shorter than its input
cycle unless the map is injective. Composition multiplies shift
parameters modulo $n$. A cyclic permutation has a unique shift parameter.

The two-sided condition $S^kF=FS^k$ used for $k$-rotation-symmetric
S-boxes in~\cite{Kavut12} is different. When $k$ is a unit, it is
equivalent to $SF=FS$. For nonunits it need not give a global shift
parameter. For example, a matrix of the form
\[
 A=\begin{pmatrix}a&b&c&d\\e&f&g&h\\c&d&a&b\\g&h&e&f\end{pmatrix}
\]
commutes with $S^2$, but generally neither $A$ nor its transpose is
$k$-circulant for any $k$.

A matrix is \emph{$k$-circulant} when each row is the preceding row
shifted $k$ positions to the right. A $1$-circulant matrix is an ordinary circulant.

\begin{lemma}
\label{lem:matrix}
For a matrix $A$ over $\F_2$, the transpose $A^{\mathsf T}$ is
$k$-circulant if and only if $AS=S^kA$.
If $A$ is invertible, $k$ is a unit modulo $n$ and
$SA=AS^m$ where $mk\equiv1\pmod n$.
Moreover, $A\bone=\bone$ and the first row of $A^{\mathsf T}$ has
odd weight and a cycle of length $n$.
\end{lemma}

\begin{proof}
The equality $AS=S^kA$ says that successive columns of $A$ are related
by $S^k$, which is precisely the row condition on $A^{\mathsf T}$.
If $A$ is invertible, conjugacy preserves order. Thus $S^k$ has order $n$
and $k$ is a unit. The inverse congruence gives $SA=AS^m$.
The vector $A\bone$ is fixed by $S^k$ and is nonzero, so it is $\bone$.
Applying the same argument to the transpose, whose shift parameter is $m$,
shows that the column sums of $A$ are one. A shorter cycle of a column
would repeat a column and contradict invertibility.
\end{proof}

An invertible matrix in the lemma is called a \emph{cyclic matrix}.
We do not claim that every cycle-preserving linear map has this form.
An affine transformation $x\mapsto Ax+c$ is $k$-shift-invariant if and
only if $x\mapsto Ax$ is $k$-shift-invariant and $S^kc=c$.
For a unit parameter the latter condition is
$c\in\{\bzero,\bone\}$.

Let $F$ be $k$-shift-invariant and $\gcd(k,n)=1$.
Then $F$ is completely determined by a single Boolean function $f$ by
\[
 F=(f,f\circ S^{-m},\ldots,f\circ S^{-(n-1)m}),
\]
where $m$ is the inverse of $k$ modulo~$n$.
If $A$ is a cyclic matrix with $AS=S^kA$ and $mk\equiv1\pmod n$,
each column is the preceding column shifted $k$ positions down. Each row
is the preceding row shifted $m$ positions to the right.


\section{Affine and cyclic affine equivalence}\label{sec:equiv}

Two functions $F,G\colon\F_2^n\to\F_2^n$ are \emph{affinely equivalent}
if there are invertible matrices $A,B$ and vectors $d,e$ with
$F(Ax+e)=BG(x)+d$. We call the pair of affine transformations
$(x\mapsto Ax+e,x\mapsto Bx+d)$ an \emph{affine witness}.
The following lemma describes which properties are preserved when the
matrices are cyclic and the constants are fixed by the shift.

\begin{lemma}
\label{lem:preserve}
Suppose $F(Ax+e)=BG(x)+d$, where $A,B$ are invertible cyclic matrices of the same shift parameter~$k$ and $d,e\in\{\bzero,\bone\}$. Then, for every $1\le m\le n$, $F$ is $m$-shift-invariant if and only if $G$ is, $F$ is cycle-preserving if and only if $G$ is, and $F$ is a permutation if and only if $G$ is.
\end{lemma}

\begin{proof}
If $F$ is $m$-shift-invariant, then
\[
\begin{split}
S^mG(x)+B^{-1}d
 &= S^mB^{-1}F(Ax+e)=B^{-1}S^{km}F(Ax+e)\\
 &= B^{-1}F(S^kAx+e)=B^{-1}F(ASx+e)\\
 &= G(Sx)+B^{-1}d.
\end{split}
\]
Thus $G$ is $m$-shift-invariant, and the same argument applies in the
opposite direction using the inverse affine transformations.
The cycle-preserving and permutation properties follow by composition,
since the two affine transformations and their inverses are cycle-preserving
permutations.
\end{proof}

We now define the restricted equivalence relation.

\begin{definition}
\label{def:cyclic-eq}
Two functions $F,G\colon\F_2^n\to\F_2^n$ are \emph{cyclically affine equivalent} if $F(Ax+e)=BG(x)+d$ for some invertible cyclic matrices $A,B$ and $d,e\in\{\bzero,\bone\}$.
We call the pair $(x\mapsto Ax+e,x\mapsto Bx+d)$ a
\emph{cyclic affine witness}.
\end{definition}

Cyclic affine equivalence implies affine equivalence.
Let $F$ and $G$ be two cyclically affine equivalent bijections, and let
$A$ and $B$ be the two implementing matrices.
Suppose that $F,G,A,B$ have shift parameters $j,k,\ell,m$, respectively.
Then comparing $F(ASx+e)$ in two ways gives $j\ell\equiv km\pmod n$.
Thus $A$ and $B$ have the same shift parameter if and only if $F$ and $G$ do.
In particular, when $F$ and $G$ are rotation-symmetric the matrices $A,B$
necessarily have the same shift parameter, and Lemma~\ref{lem:preserve} applies.

If two shift-invariant bijections $F$ and $G$ are affinely equivalent,
what can we say about their generating Boolean functions $f$ and $g$?
The next result gives a precise condition for cyclic affine equivalence.

\begin{theorem}
\label{thm:main}
Let $f,g$ induce rotation-symmetric permutations $F,G$ on $\F_2^n$. Then $F$ and $G$ are cyclically affine equivalent if and only if there exist an invertible cyclic matrix $A$, a vector $b=(b_j)_{j=1}^n$ for which the cyclic matrix with first row~$b$ and the same shift parameter as~$A$ is invertible, a bit $d'\in\{0,1\}$, and $e\in\{\bzero,\bone\}$ such that
\begin{equation}\label{eq:main}
\sum_{j=1}^{n} b_j\,(g\circ S^{1-j})(x) + d' \;=\; f(Ax+e)\qquad\text{for all }x\in\F_2^n.
\end{equation}
For $k$-shift-invariant permutations $F,G$ with $\gcd(k,n)=1$, the same holds with $S^{1-j}$ replaced by $S^{(1-j)m}$, where $mk\equiv1\pmod n$.
\end{theorem}

\begin{proof}
Write $AS=S^\ell A$ and choose $t$ with $\ell t\equiv1\pmod n$.
Let $B$ be the unique matrix with first row $b$ and $BS=S^\ell B$.
For rotation-symmetric $F,G$, the two maps
$H(x)=F(Ax+e)$ and $K(x)=BG(x)+d'\bone$ both satisfy
$HS=S^\ell H$ and $KS=S^\ell K$. Their $i$-th coordinates are
$H_1(S^{(1-i)t}x)$ and $K_1(S^{(1-i)t}x)$.
Consequently they agree exactly when their first coordinates agree.
These first coordinates give~\eqref{eq:main}.
Necessity follows by taking the first row of an implementing matrix $B$.
For maps with common unit parameter $k$, replace $\ell$ by $k\ell$
in this coordinate argument and use
$G_j=g\circ S^{(1-j)m}$, where $mk\equiv1\pmod n$.
\end{proof}

Writing $g_b=\sum_{j=1}^n b_j(g\circ S^{1-j})+d'$,
equation~\eqref{eq:main} says that $f$ and $g_b$ are affinely equivalent
through a cyclic matrix with a shift-fixed translation. If $b$ is the
first standard basis vector, then $g_b=g+d'$. In general the criterion
relates $f$ to a linear combination of shifts of $g$, rather than to $g$
itself.

If one of the two permutations is affine, the two equivalence relations coincide.

\begin{lemma}
\label{lem:affine-case}
Let $F$ be an affine $j$-shift-invariant permutation, and let $G$ be a $k$-shift-invariant permutation affinely equivalent to $F$, where $\gcd(jk,n)=1$. Then $F$ and $G$ are cyclically affine equivalent.
\end{lemma}

\begin{proof}
Write $F(x)=Dx+b$. Its shift-invariance gives that $D$ is an invertible cyclic matrix and $b\in\{\bzero,\bone\}$. Affine equivalence to $F$ implies that $G$ is affine; write $G(x)=Ex+c$. The shift-invariance of $G$ likewise makes $E$ an invertible cyclic matrix and $c\in\{\bzero,\bone\}$. The invertible cyclic matrices form a group, so $A=D^{-1}E$ is invertible and cyclic. Taking $B=I$, $e=\bzero$, and $d=b+c$, we obtain
\[
F(Ax+e)=D(D^{-1}Ex)+b=Ex+b=G(x)+b+c=BG(x)+d. \qedhere
\]
\end{proof}


\section{Dimension classification}\label{sec:dichotomy}

We determine exactly when affine equivalence implies cyclic affine
equivalence for every pair of cyclic permutations in a given dimension.
A conjugacy criterion and Sylow theory prove the implication for odd
primes $p$ with $\ord_p(2)=p-1$. Three constructions show that it fails
in all other dimensions $n\ge3$ and quantify the number of cyclic affine
classes in one affine class for odd $n$ when the $n$th cyclotomic polynomial
$\Phi_n$ has more than one irreducible factor over $\F_2$.

\relax
An exact seven-dimensional family follows the classification.

\relax
Write
\[
\mathcal P_n=\{F\colon\F_2^n\to\F_2^n\ \text{a permutation with }FS=S^kF
\text{ for some }k\text{ coprime to }n\}.
\]
Both equivalence relations restrict to $\mathcal P_n$. Classes in this
section refer to these restrictions.
\relax

The seven-dimensional family below uses explicitly stated class intersections.
\relax

Non-bijective maps and
permutations without a global shift parameter are outside the classification.

\subsection{The normalizer of the rotation group}
\label{subsec:normalizer}

\begin{lemma}\label{lem:normalizer}
The normalizer $N_{\AGL_n(\F_2)}(\langle S\rangle)$ is the set of maps
$x\mapsto Ax+e$ with $A$ an invertible cyclic matrix and
$e\in\{\bzero,\bone\}$.  Two functions $F,G$ are cyclically affine equivalent
if and only if $F\alpha=\beta G$ for some $\alpha,\beta$ in this normalizer.
\end{lemma}

\begin{proof}
Let $\alpha(x)=Ax+e$ with $A$ invertible.  Then $\alpha^{-1}(y)=A^{-1}y+A^{-1}e$
and
\[
\alpha S\alpha^{-1}(y)=ASA^{-1}y+ASA^{-1}e+e.
\]
Suppose $\alpha$ normalizes $\langle S\rangle$.  Conjugation is an
automorphism of a cyclic group of order $n$, so $\alpha S\alpha^{-1}=S^t$ with
$\gcd(t,n)=1$.  Comparing linear parts gives $ASA^{-1}=S^t$, which makes $A$
an invertible cyclic matrix of shift parameter $t$ by
Lemma~\ref{lem:matrix}.  Comparing constants gives $(I+S^t)e=\bzero$, so
$e\in\{\bzero,\bone\}$, since a shift by $t$
permutes the coordinates in one cycle. Conversely both conditions give
$\alpha S\alpha^{-1}=S^t$, and $\langle S^t\rangle=\langle S\rangle$.  The
last assertion is Definition~\ref{def:cyclic-eq}.
\end{proof}

The normalizer has order $2c_n\varphi(n)$ in the notation of
Section~\ref{sec:counting}.

\subsection{A conjugacy criterion}\label{subsec:fusion}

For cyclic permutations, the next lemma reduces cyclic affine equivalence
to conjugacy of $\langle S\rangle$ and
$\alpha\langle S\rangle\alpha^{-1}$ in the affine self-equivalence
group $\mathcal A_F$ defined below. The criterion holds in every dimension.

\begin{lemma}\label{lem:fusion}
Let $F,G$ be $j$- and $k$-shift-invariant permutations of $\F_2^n$ and let
$F\alpha=\beta G$ with $\alpha,\beta\in\AGL_n(\F_2)$.  Put
\[
\mathcal A_F=\{\gamma\in\AGL_n(\F_2):F\gamma F^{-1}\in\AGL_n(\F_2)\},
\qquad S_\alpha=\alpha S\alpha^{-1}.
\]
Then $\mathcal A_F$ is a subgroup of $\AGL_n(\F_2)$ containing
$\langle S\rangle$ and $\langle S_\alpha\rangle$, both of order $n$, and the
following hold.
\begin{itemize}
\item[(i)] If $F$ and $G$ are cyclically affine equivalent, then
$\langle S_\alpha\rangle$ and $\langle S\rangle$ are conjugate in
$\mathcal A_F$.
\item[(ii)] If $\gcd(jk,n)=1$ and $\langle S_\alpha\rangle$ and
$\langle S\rangle$ are conjugate in $\mathcal A_F$, then $F$ and $G$ are
cyclically affine equivalent.
\end{itemize}
The $\mathcal A_F$-conjugacy class of $\langle S_\alpha\rangle$ does not depend
on the choice of $\alpha$.
\end{lemma}

\begin{proof}
If $F\gamma F^{-1}$ and $F\delta F^{-1}$ are affine, then so are
$F\gamma\delta F^{-1}$ and $F\gamma^{-1}F^{-1}=(F\gamma F^{-1})^{-1}$, because
the inverse of an affine permutation is affine.  So $\mathcal A_F$ is a group.
From $FS=S^jF$ we get $FSF^{-1}=S^j$, hence $\langle S\rangle\le\mathcal A_F$.
From $GS^i=S^{ki}G$ we get
\[
F(\alpha S^i\alpha^{-1})F^{-1}=(F\alpha)S^i(F\alpha)^{-1}
=(\beta G)S^i(\beta G)^{-1}=\beta S^{ki}\beta^{-1},
\]
which is affine, hence $\langle S_\alpha\rangle\le\mathcal A_F$.  Both
subgroups have order $n$, since $S_\alpha$ is conjugate to $S$.

For (i), let $(\alpha',\beta')$ be a cyclic affine witness, so
$F\alpha'=\beta'G$ with $\alpha',\beta'$ in the normalizer of
Lemma~\ref{lem:normalizer}.  Put $\gamma=\alpha'\alpha^{-1}$.  Then
\[
F\gamma F^{-1}=(F\alpha')(F\alpha)^{-1}=\beta'G(\beta G)^{-1}=\beta'\beta^{-1}
\]
is affine, so $\gamma\in\mathcal A_F$, and
$\gamma S_\alpha\gamma^{-1}=\alpha'S\alpha'^{-1}$ generates $\langle S\rangle$.

For (ii), let $\gamma\in\mathcal A_F$ with
$\gamma\langle S_\alpha\rangle\gamma^{-1}=\langle S\rangle$.  Put
$\alpha'=\gamma\alpha$ and $\beta'=F\alpha'G^{-1}$.  Then
$\alpha'\langle S\rangle\alpha'^{-1}=\langle S\rangle$, so
$\alpha'S\alpha'^{-1}=S^t$ with $\gcd(t,n)=1$.  Also
\[
\beta'=F\gamma\alpha G^{-1}=(F\gamma F^{-1})(F\alpha G^{-1})
=(F\gamma F^{-1})\beta
\]
is a product of two affine permutations, so $F\alpha'=\beta'G$ is again an
affine equivalence.  From $G^{-1}S^kG=S$,
\[
\beta'S^k\beta'^{-1}=F\alpha'G^{-1}S^kG\alpha'^{-1}F^{-1}
=F\alpha'S\alpha'^{-1}F^{-1}=FS^tF^{-1}=S^{jt}.
\]
Here $S^k$ generates $\langle S\rangle$ because $\gcd(k,n)=1$, and $S^{jt}$ has
order $n$ because $\gcd(jt,n)=1$.  So $\beta'$ normalizes $\langle S\rangle$
as well, and Lemma~\ref{lem:normalizer} makes $(\alpha',\beta')$ a cyclic
affine equivalence.

For another witness $F\alpha_1=\beta_1G$, the same calculation gives
$\alpha_1\alpha^{-1}\in\mathcal A_F$; this element conjugates $S_\alpha$
to $S_{\alpha_1}$, proving witness independence.
\end{proof}

For rotation-symmetric $F,G$, the witnesses produced in (ii) have the same
shift parameter.

\subsection{Cyclic \texorpdfstring{$p$}{p}-subgroups}\label{subsec:sylow}

\begin{lemma}\label{lem:cyclic-p}
Let $p$ be an odd prime with $\ord_p(2)=p-1$.  Then every $p$-subgroup of
$\AGL_p(\F_2)$ is cyclic.
\end{lemma}

\begin{proof}
The translation subgroup $T\cong(\Z/2)^p$ is normal in $\AGL_p(\F_2)$ with
quotient $\GL_p(\F_2)$.  A $p$-subgroup $Q$ meets $T$ trivially, since
$Q\cap T$ is both a $p$-group and a $2$-group.  So the quotient map is
injective on $Q$ and it is enough to treat $\GL_p(\F_2)$.  Now
\[
|\GL_p(\F_2)|=2^{p(p-1)/2}\prod_{i=1}^{p}(2^i-1),
\]
and $p\mid2^i-1$ exactly when $(p-1)\mid i$.  For $1\le i\le p$ this happens
only at $i=p-1$, because $(p-1)\mid p$ would force $(p-1)\mid1$.  Hence the
$p$-part of $|\GL_p(\F_2)|$ is $p^a$ with $p^a$ the $p$-part of $2^{p-1}-1$.
Multiplication by $\F_{2^{p-1}}^\times$ on
$\F_{2^{p-1}}\cong\F_2^{p-1}$ embeds a cyclic group of order $2^{p-1}-1$ into
$\GL_{p-1}(\F_2)$, and extending by the identity on a complementary line
embeds it into $\GL_p(\F_2)$.  Its $p$-part is cyclic of order $p^a$, so it is
a Sylow $p$-subgroup.  All Sylow $p$-subgroups are conjugate, hence all are
cyclic, hence every $p$-subgroup is cyclic.
\end{proof}

\begin{remark}\label{rem:sylow-context}
The cyclic Sylow argument needs no hypothesis on the $p$-adic valuation of
$2^{p-1}-1$. The cyclicity of Sylow $p$-subgroups of
$\GL_d(q)$ for a primitive prime divisor $p$ of $q^d-1$ is classical, as described by
Weir~\cite{Weir55}.
\end{remark}

\subsection{The dichotomy}\label{subsec:dichotomy}

For any dimension, a permutation satisfying $FS=S^kF$ has a unit
shift parameter, since $FSF^{-1}=S^k$ forces $S^k$ to have order $n$.
Thus every permutation with a global shift parameter belongs to
$\mathcal P_n$.

\begin{theorem}\label{thm:prime-dichotomy}
Let $p$ be an odd prime.  The following conditions are equivalent.
\begin{itemize}
\item[(i)] $\ord_p(2)=p-1$, equivalently $1+X+\dotsb+X^{p-1}$ is irreducible
over $\F_2$.
\item[(ii)] any two affinely equivalent cyclic permutations of
$\F_2^p$ are cyclically affine equivalent.
\end{itemize}
\end{theorem}

\begin{proof}
Assume (i) and let $F\alpha=\beta G$ with $F,G$ cyclic permutations of
$\F_2^p$ and $\alpha,\beta\in\AGL_p(\F_2)$.  Both shift parameters are units,
so Lemma~\ref{lem:fusion} applies and $\langle S\rangle$ and
$\langle S_\alpha\rangle$ are subgroups of $\mathcal A_F$ of order $p$.
Choose Sylow $p$-subgroups $P_1\ge\langle S\rangle$ and
$P_2\ge\langle S_\alpha\rangle$ of $\mathcal A_F$, and
$\delta\in\mathcal A_F$ with $\delta P_2\delta^{-1}=P_1$.  By
Lemma~\ref{lem:cyclic-p} the group $P_1$ is cyclic, so it has a unique
subgroup of order $p$.  Both $\langle S\rangle$ and
$\delta\langle S_\alpha\rangle\delta^{-1}$ are subgroups of $P_1$ of order
$p$, hence they are equal.  Lemma~\ref{lem:fusion}(ii) now gives a cyclic
affine witness, which is (ii).

If (i) fails, then Proposition~\ref{prop:split-witness} below produces two
affinely equivalent cyclic permutations of $\F_2^p$ that are not
cyclically affine equivalent, so (ii) fails.
\end{proof}

\subsection{Separation constructions}\label{subsec:separation-constructions}

We now construct pairs that separate the two equivalence relations.
The following lemma turns affine equivalence of permutations with small
moved sets into affine conjugacy.

\begin{lemma}\label{lem:small-support}
Let $F,G$ be permutations of $\F_2^n$ that equal the identity outside sets
$U_F,U_G$, respectively.  If
$|U_F|+|U_G|<2^{n-1}$ and
$F\alpha=\beta G$ for $\alpha,\beta\in\AGL_n(\F_2)$, then
$\alpha=\beta$.
\end{lemma}

\begin{proof}
For every $x\notin U_G$ with $\alpha(x)\notin U_F$, we have
$\alpha(x)=F\alpha(x)=\beta G(x)=\beta(x)$.  Thus the two affine maps agree
on more than $2^{n-1}$ points.  The zero set of a nonzero affine map from
$\F_2^n$ to $\F_2^n$ is either empty or an affine subspace of size at most
$2^{n-1}$, so the maps are equal.
\end{proof}

\subsubsection{Even dimensions}\label{subsec:even-separation}

Let $\operatorname{par}(x)$ denote the
parity of $x$, put $J=\bone\bone^{\mathsf T}$, and let $\delta_Z$ be the
indicator of a set $Z\subseteq\F_2^n$.  If $Z$ is invariant under both $S$ and
$x\mapsto x+\bone$, then
\[
T_Z(x)=x+\delta_Z(x)\bone
\]
is a rotation-symmetric involution.

\begin{theorem}\label{thm:even-separation}
Let $n\ge4$ be even, let $w=0101\cdots01$, and put $\overline w=w+\bone$.
Define
\[
X=\{w,\overline w\},\qquad
Y=\{x\in\F_2^n:\operatorname{par}(x)=0\}\setminus\{\bzero,\bone\},
\]
and let $F=T_X$ and $G=T_Y$.  Then $F$ and $G$ are rotation-symmetric
permutations and
\begin{equation}\label{eq:even-witness}
F(x+w)=(I+J)G(x)+\overline w\qquad\text{for all }x\in\F_2^n.
\end{equation}
Thus $F$ and $G$ are affinely equivalent but not cyclically affine equivalent.
\end{theorem}

\begin{proof}
The sets $X$ and $Y$ are invariant under shift and complementation, so $F$ and
$G$ are rotation-symmetric involutions.  Since $n$ is even, $J^2=0$ and hence
$I+J$ is an invertible circulant.  Using
$Jx=\operatorname{par}(x)\bone$, identity~\eqref{eq:even-witness} reduces to
\[
\delta_Y(x)=1+\operatorname{par}(x)+
\delta_{\{\bzero,\bone\}}(x),
\]
which follows by considering $x=\bzero$, $x=\bone$, the remaining
even-weight vectors, and the odd-weight vectors. This proves affine equivalence.

Suppose, toward a contradiction, that
$F(Ax+e)=BG(x)+d$ is a cyclic affine equivalence with
$e,d\in\{\bzero,\bone\}$.  Every invertible cyclic matrix fixes $\bone$, so
expanding the two $T$-maps and using $B\bone=\bone$ shows that
\[
(A+B)x+(e+d)+
\bigl(\delta_X(Ax+e)+\delta_Y(x)\bigr)\bone=0.
\]
Thus the image of $A+B$ is contained in $\langle\bone\rangle$, and we may
write $A+B=\bone\lambda^{\mathsf T}$. The matrices $A,B$ have the
same shift parameter $k$, as observed after Definition~\ref{def:cyclic-eq}.
Thus $(A+B)S=S^k(A+B)$ gives $\lambda^{\mathsf T}S=\lambda^{\mathsf T}$,
so $A+B$ is either $0$ or $J$. It follows that, for some
$\mu,\varepsilon\in\{0,1\}$,
\[
\delta_Y(x)=\delta_X(Ax+e)+\mu\operatorname{par}(x)+\varepsilon.
\]
Putting $x=\bzero$ gives $\varepsilon=0$.  If $\mu=0$, the equality is
impossible because $|X|=2$ while $|Y|=2^{n-1}-2$.  If $\mu=1$, choose an
odd-weight $x$ outside the at most two preimages of $X$ under $x\mapsto Ax+e$.
Then the right-hand side equals one while $\delta_Y(x)=0$, again a
contradiction.
\end{proof}

\subsubsection{Square factors}\label{subsec:square-factor}

The next construction handles dimensions divisible by a square $d^2$
with $d\ge3$. In particular, it covers every odd prime power with exponent
at least two. Each of the two permutations differs from the identity on
just two shift cycles.

\begin{theorem}\label{thm:square-factor}
Let $n$ have a divisor $d\ge3$ with $d^2\mid n$.  Then there are
rotation-symmetric permutations $F,G\colon\F_2^n\to\F_2^n$ and a linear
permutation $P$ with
\[
FP=PG,
\]
such that $F$ and $G$ are not cyclically affine equivalent. Each of $F$ and $G$ differs from the
identity on exactly two $S$-cycles, of lengths $n$ and $d$.
\end{theorem}

\begin{proof}
Identify $\F_2^n$ with $R_n=\F_2[X]/(X^n-1)$, so that $S$ is multiplication
by $X$. Put
\[
 C=c(1)=\{X^j:0\le j<n\},\qquad
 v=\sum_{q=0}^{n/d-1}X^{dq},\qquad D=c(v),
\]
and define
\[
G(x)=\begin{cases}S^dx,&x\in C,\\ Sx,&x\in D,\\ x,&\text{otherwise,}\end{cases}
\qquad
F(x)=\begin{cases}S^dx,&x\in C,\\ S^{-1}x,&x\in D,\\ x,&\text{otherwise.}\end{cases}
\]
The vectors $X^rv$, for $0\le r<d$, have pairwise disjoint supports of size
$n/d\ge d\ge3$. Thus $C$ and $D$ are disjoint shift cycles of lengths
$n$ and $d$. Both maps are rotation-symmetric permutations and move exactly
$C\cup D$. They differ on $D$ since $d\ge3$.

Let $P$ be the coordinate permutation $dq+r\mapsto dq+(-r\bmod d)$ for
$0\le r<d$. It preserves $C$ and $D$, commutes with $S^d$, and satisfies
$P(X^rv)=X^{-r}v$. Hence $FP=PG$ on $C$. On $D$,
\[
 PG(X^rv)=X^{-r-1}v=F(X^{-r}v)=FP(X^rv).
\]
Outside $C\cup D$ both sides equal $P$. This proves linear equivalence.

Suppose $F\alpha=\beta G$ is a cyclic affine equivalence.
Since $n\ge9$ and $2(n+d)<2^{n-1}$, Lemma~\ref{lem:small-support} gives
$\alpha=\beta$. Write $\alpha S=S^\ell\alpha$ with $\gcd(\ell,n)=1$.
Conjugacy implies $\alpha(C\cup D)=C\cup D$.
A cyclic affine permutation preserves shift-orbit lengths, so
$\alpha(C)=C$ and $\alpha(D)=D$.
On $C$, the identity $F\alpha=\alpha G$ gives
$S^d\alpha(x)=S^{d\ell}\alpha(x)$. Its points have shift period $n$, hence
\[
 \ell\equiv1\pmod{n/d}.
\]
On $D$ the same identity gives $S^{-1}\alpha(x)=S^\ell\alpha(x)$, hence
\[
 \ell\equiv-1\pmod d.
\]
But $d^2\mid n$ implies $d\mid n/d$, so the first congruence also gives
$\ell\equiv1\pmod d$. This contradicts $d\ge3$.
\end{proof}

\begin{example}\label{ex:nine}
Take $n=9$, $d=3$, and $v=1+X^3+X^6$, so $D=\{v,Xv,X^2v\}$ consists of three vectors of
weight three, and $C$ is the set of nine monomials.  Then $G$ acts as $S^3$ on
$C$ and as $S$ on $D$, while $F$ acts as $S^3$ on $C$ and as $S^{-1}$ on $D$.
Each moves exactly $12$ of the $512$ points. Although all nontrivial
cycles of $F$ and $G$ have length three, their moved sets consist of shift
cycles of distinct lengths nine and three. A cyclic affine conjugator must
preserve these shift cycles, forcing the incompatible congruences
$\ell\equiv1$ and $\ell\equiv-1\pmod3$.
\end{example}

\subsubsection{Reducible cyclotomic factors}\label{subsec:split-construction}

For odd $n$, the final construction applies whenever the cyclotomic
polynomial $\Phi_n$ splits into more than one irreducible factor over $\F_2$.

\begin{proposition}\label{prop:split-witness}
Let $n\ge7$ be odd with $d=\ord_n(2)<\varphi(n)$. There are
rotation-symmetric permutations $F_0,\ldots,F_{d-1}$ of $\F_2^n$ that are
pairwise linearly conjugate and pairwise not cyclically affine equivalent.
Each differs from the identity on exactly two $S$-cycles of length $n$.
Thus some affine class contains at least $d$ cyclic affine classes.
\end{proposition}

\begin{proof}
Identify $\F_2^n$ with $R_n=\F_2[X]/(X^n-1)$, so that $S$ is multiplication by
$X$. Since $n$ is odd, $R_n$ is a product of fields. The irreducible factors
of $\Phi_n$ over $\F_2$ all have degree $d=\ord_n(2)$, so there are
$\varphi(n)/d\ge2$ of them. Choose the primitive idempotents $e,f\in R_n$
associated with two distinct such factors in the Chinese remainder
decomposition. Each has coordinate $1$ in
its selected field component and $0$ in all other components. Put
\[
C_e=\{X^je:0\le j<n\},\qquad C_f=\{X^jf:0\le j<n\}.
\]
The image of $X$ in $eR_n\cong\F_{2^d}$ has multiplicative order $n$, so $C_e$
is a full $S$-cycle of length $n$, and likewise $C_f$. They are disjoint,
since $eR_n\cap fR_n=0$ and neither contains $\bzero$.

For $0\le j<d$, let $F_j$ act as $S^{2^j}$ on $C_e$, as $S$ on $C_f$,
and as the identity elsewhere. Both sets are $S$-invariant and
$\gcd(2^j,n)=1$, so every $F_j$ is a rotation-symmetric permutation with
exactly two nontrivial cycles, both of length $n$.

In characteristic two the Frobenius $\sigma_2$ is both the squaring map of
$R_n$ and the coordinate permutation $i\mapsto2i\pmod n$. It fixes the
idempotent $e$. With $M_z$ denoting multiplication by $z$, put
\[
P_j=M_e\sigma_2^j+M_{1+e}.
\]
This applies a Frobenius power on $eR_n$ and fixes $(1+e)R_n$, so it is
linear and invertible. On $C_e$ both $F_jP_j$ and $P_jF_0$ send $X^re$ to
$X^{2^j(r+1)}e$. On $C_f$ both send $X^rf$ to $X^{r+1}f$.
Since $P_j$ preserves $C_e\cup C_f$, both compositions equal $P_j$ on its
complement. Hence $F_jP_j=P_jF_0$, proving linear conjugacy.

Suppose $F_i\alpha=\beta F_j$ is a cyclic affine equivalence. Since
$4n<2^{n-1}$, Lemma~\ref{lem:small-support} gives $\alpha=\beta$.
Thus $\alpha$ conjugates $F_j$ to $F_i$ and either preserves or exchanges
their two nontrivial cycles. Write $\alpha(x)=Ax+t$ with
$\alpha S\alpha^{-1}=S^\ell$ and $t\in\{\bzero,\bone\}$.
If $F_j$ acts as $S^s$ on a cycle $C$ and $F_i$ acts as $S^{s'}$ on
$\alpha(C)$, then for $x\in C$ we have
$S^{s'}\alpha(x)=F_i(\alpha(x))=\alpha S^s(x)=S^{\ell s}\alpha(x)$.
Since $\alpha(C)$ has shift period $n$, we obtain
$\ell s\equiv s'\pmod n$.
Preserving both cycles gives $\ell\equiv1$ and $2^j\equiv2^i\pmod n$,
hence $i=j$.

Exchanging the cycles instead forces $\ell\equiv2^i\pmod n$.
The relation $AS=S^\ell A$ then gives
\[
 A(ex)=e(X^\ell)A(x)=e(X^{2^i})A(x)=e^{2^i}A(x)=eA(x).
\]
To justify this identity, expand $e$ as a polynomial in $X$ and use
$AS^r=S^{\ell r}A$ term by term. The equality
$e(X^{2^i})=e^{2^i}=e$ follows because $e$ is idempotent.
Thus $A$ preserves $eR_n$. Also $f\bone=0$, since the all-ones vector
lies in the $X-1$ component of the same decomposition, whereas $f$
belongs to a primitive $n$-th-root component. Since $ef=0$, it follows
that $f\alpha(x)=0$ for $x\in C_e$. But $fy=y\ne0$ for every $y\in C_f$.
This excludes the exchange and proves pairwise cyclic affine inequivalence.
\end{proof}


\subsection{The classification of the dimensions}\label{subsec:five-case}

The preceding results give four disjoint cases, collected in
Table~\ref{tab:dimensions}. To see that they exhaust all $n\ge3$, suppose
$n$ is odd and $\ord_n(2)=\varphi(n)$. Then its unit group is cyclic.
If $n$ had two distinct prime divisors, the Chinese remainder theorem
would give at least four solutions of $u^2=1$ in that group, which a cyclic
group cannot have. Thus $n$ is a prime power.

\begin{table}[ht]
\caption{Whether affine equivalence implies cyclic affine equivalence on
cyclic permutations of $\F_2^n$. The four rows are disjoint and cover all
$n\ge3$. Each row marked ``no'' supplies an explicit separating pair.
In the last two rows, $p$ is an odd prime.}
\label{tab:dimensions}
{\footnotesize
\begin{tabular}{lcl}
\hline
dimension $n$ & implies & source\\
\hline
even, $n\ge4$ & no & Theorem~\ref{thm:even-separation}\\
odd, $\ord_n(2)<\varphi(n)$ & no & Proposition~\ref{prop:split-witness}\\
\shortstack[l]{$n=p^a$, $a\ge2$,\\$\ord_n(2)=\varphi(n)$}
 & no & Theorem~\ref{thm:square-factor}\\
$n=p$, $\ord_p(2)=p-1$ & \textbf{yes} & Theorem~\ref{thm:prime-dichotomy}\\
\hline
\end{tabular}
}
\end{table}

\begin{corollary}\label{cor:classification}
Let $n\ge3$.\footnote{At $n=2$ every permutation of $\F_2^2$ is affine.
Lemma~\ref{lem:affine-case} therefore makes the implication true on cyclic
permutations as well.} On cyclic permutations of $\F_2^n$, affine equivalence
implies cyclic affine equivalence if and only if $n$ is an odd prime for which
$2$ is a primitive root.
\end{corollary}

\begin{proof}
For an odd prime $n$ with $\ord_n(2)=n-1$, the implication follows from
Theorem~\ref{thm:prime-dichotomy}. If $n\ge4$ is even,
Theorem~\ref{thm:even-separation} gives a separating pair. If $n$ is odd
and $\ord_n(2)<\varphi(n)$, then $n\ge7$ and
Proposition~\ref{prop:split-witness} applies. In the remaining case,
$n=p^a$ for an odd prime $p$ and $a\ge2$, so
Theorem~\ref{thm:square-factor} applies with $d=p$.
\end{proof}

Thus, for $n\ge3$, one representative of each affine class in
$\mathcal P_n$ also gives one representative of each cyclic affine class
exactly when $n$ is an odd prime for which $2$ is a primitive root.
At every other odd prime $p$, some affine class contains at least
$\ord_p(2)$ cyclic affine classes by
Proposition~\ref{prop:split-witness}.

\begin{remark}\label{rem:artin}
The odd primes for which the two relations coincide begin
\[3,5,11,13,19,29,37,53,59,61,67,83.\]
The odd primes for which the two relations differ begin
\[7,17,23,31,41,43,47,71,73,79,89,97.\]
Artin's conjecture predicts density $0.3739558136\dotsc$ among the primes
for the first set~\cite{A005596}. Hooley proved this under the generalized Riemann
hypothesis~\cite{Hooley67}. Heath-Brown's unconditional result permits
at most two exceptional prime bases, but does not settle the prescribed
base $2$~\cite{HB86}.
An unconditional proof of infinitude for every admissible integer base,
including $2$, was announced in October 2026~\cite{OpenAI26}.
We treat this as an unrefereed announcement and do not use it in any proof.
If it is validated, there are infinitely many dimensions in which the
two relations coincide.
The dimension criterion itself is unconditional and independent of the
size of that set of primes.
\end{remark}

\relax
\subsection{The exact seven-dimensional family}\label{subsec:p7-family}

The first odd dimension in which the two equivalence relations differ is
$7$. In this dimension the construction of Proposition~\ref{prop:split-witness}
extends to a small family whose class intersections can be determined exactly.
An $S$-invariant subspace of $R_7=\F_2[X]/(X^7-1)$ is an ideal.
Since $\ord_7(2)=3$ we have
$R_7\cong\F_2\times\F_8\times\F_8$.
Let $V_1$ and $V_2$ be its two ideals of dimension three.
They satisfy
\[
R_7=\F_2\bone\oplus V_1\oplus V_2.
\]
Each $V_i\setminus\{\bzero\}$ is one $S$-cycle of length seven.
For $(a,b)\in(\Z/7)^2$ define
\begin{equation}\label{eq:p7-family}
F_{a,b}(x)=
\begin{cases}
S^ax,&x\in V_1,\\
S^bx,&x\in V_2,\\
x,&\text{otherwise.}
\end{cases}
\end{equation}
The first two cases agree at $\bzero$.
Write $\mathcal F=\{F_{a,b}:a,b\in\Z/7\}$ and
$H_2=\langle2\rangle=\{1,2,4\}\le(\Z/7)^\times$.
Its other coset is $3H_2=\{3,5,6\}$.
We use $H_2\,0=3H_2\,0=\{0\}$.

\begin{proposition}\label{prop:p7-split}
The $49$ members of $\mathcal F$ are distinct rotation-symmetric
permutations.  Their equivalence relations are as follows.
\begin{itemize}
\item[(i)] $F_{a,b}$ and $F_{c,d}$ are affinely equivalent if and only if
\begin{equation}\label{eq:p7-affine-tags}
(c,d)\in(H_2a\times H_2b)\cup(3H_2b\times3H_2a).
\end{equation}
Every such equivalence can be realized by linear conjugation.
\item[(ii)] $F_{a,b}$ and $F_{c,d}$ are cyclically affine equivalent if and
only if for some $\ell\in(\Z/7)^\times$,
\begin{equation}\label{eq:p7-cyclic-tags}
(c,d)=
\begin{cases}
(\ell a,\ell b),&\ell\in H_2,\\
(\ell b,\ell a),&\ell\in3H_2.
\end{cases}
\end{equation}
\item[(iii)] The six affine classes meeting $\mathcal F$ have intersection
sizes $1,6,6,9,9,18$ with the family. The ten cyclic affine classes meeting
$\mathcal F$ have intersection sizes $1,3,3,6,6,6,6,6,6,6$.
\end{itemize}
\end{proposition}

\begin{proof}
The three pieces in \eqref{eq:p7-family} are $S$-invariant and the prescribed
maps are bijections on them.  Thus every member is a permutation commuting
with $S$. Its nontrivial cycles are precisely the sets
$V_i\setminus\{\bzero\}$ whose corresponding exponent is nonzero.
It moves at most $14$ points.
The restrictions to the two seven-cycles distinguish all $49$ members.

Suppose two family members satisfy $F_{c,d}\alpha=\beta F_{a,b}$ with affine
permutations $\alpha,\beta$.
Both maps have at most $14$ moved points.  Since $28<64$,
Lemma~\ref{lem:small-support} gives $\alpha=\beta$.
Thus $\alpha$ conjugates $F_{a,b}$ to $F_{c,d}$.
If one is the identity then both are the identity.
Otherwise write $\alpha(x)=Ax+t$ and take a nontrivial cycle $C$ of
$F_{a,b}$. Both $C$ and $\alpha(C)$ are three-dimensional subspaces
with their zero vectors removed.  Their vector sums are zero.  Since $|C|=7$,
\[
\bzero=\sum_{x\in C}\alpha(x)
       =A\sum_{x\in C}x+t=t.
\]
Consequently every affine equivalence between nonidentity members is a
linear conjugacy.  It carries each active space $V_i$ onto an active space
of the other map.  Here a space is active when its exponent is nonzero.

On $V_i\cong\F_8$ the eigenvalues of $S^a$, for $a\ne0$, are
$\lambda_i^a,\lambda_i^{2a},\lambda_i^{4a}$.
Choose $\lambda_1$ and $\lambda_2$ so that
$\lambda_2=\lambda_1^3$.
Two restrictions on the same space are linearly conjugate exactly when
their exponents differ by a factor in $H_2$.
A restriction on $V_1$ is conjugate to a restriction on $V_2$ exactly
when the two exponents differ by a factor in $3H_2$.
This follows from their irreducible characteristic polynomials.
If both spaces are active then a conjugator either preserves them or
exchanges them.  If just one is active then its image determines which of
these two cases applies to the active space.
These conditions give \eqref{eq:p7-affine-tags}.
Conversely the required linear isomorphisms on the active spaces exist by
the same characteristic-polynomial criterion.
Choose any isomorphism between the inactive spaces when there are any.
Extend their direct sum by fixing $\bone$.
This gives an invertible linear map conjugating the two family members.
It also proves (i) when an exponent is zero.

For (ii), the preceding argument applies to a cyclic affine equivalence.
Thus every cyclic affine equivalence between nonidentity members is a
linear conjugacy.
A cyclic linear map of shift parameter $\ell$ has the form
\[
x\longmapsto u\,x(X^\ell)
\]
with $u$ a unit of $R_7$.
Multiplication by $u$ preserves each ideal and commutes with each power
of $S$.  The substitution $X\mapsto X^\ell$ preserves both $V_i$ for
$\ell\in H_2$ and exchanges them for $\ell\in3H_2$.
Conjugation carries $S^a$ to $S^{\ell a}$.
Its action on the exponents is therefore exactly
\eqref{eq:p7-cyclic-tags}.
Each listed action is realized by the coordinate substitution itself.
The identity case also satisfies the stated condition.

For (iii), \eqref{eq:p7-affine-tags} gives the identity intersection and
two intersections of size $6$ with one exponent zero.
For nonzero exponents it gives one intersection of size $18$ when the
exponents lie in the same coset of $H_2$.
The two opposite-coset cases each give an intersection of size $9$.
The order-six action in \eqref{eq:p7-cyclic-tags} has the identity orbit
and two orbits of size $6$ with one exponent zero.
On the remaining pairs it has two orbits of size $3$ on $b=-a$ and five
orbits of size $6$.
These are the claimed ten intersections.
\end{proof}

In particular $F_{1,1}$ and $F_{2,1}$ are linearly equivalent and are not
cyclically affine equivalent.
A conjugator applies the Frobenius map to $V_1\cong\F_8$ and fixes
$\F_2\bone\oplus V_2$.
The pair $(2,1)$ is absent from the cyclic orbit of $(1,1)$ in
\eqref{eq:p7-cyclic-tags}.

\begin{remark}\label{rem:p7-mersenne}
At $p=7$ the two sets $V_i\setminus\{\bzero\}$ are the cycles used in
Proposition~\ref{prop:split-witness}.
For a prime $p$, the nonzero elements of a nontrivial minimal ideal form
a single $S$-cycle exactly when
$2^d-1=p$, where $d=\ord_p(2)$.
At $p=17$ a minimal ideal has $255$ nonzero points in fifteen cycles.
The construction in Proposition~\ref{prop:split-witness} uses the cycle
of its idempotent.
\end{remark}

Table~\ref{tab:p7} records the intersections of affine and cyclic affine
classes with $\mathcal F$.
The two rows with one zero exponent are represented by $(1,0)$ and $(3,0)$.
Exchanging the active ideal also exchanges the coset of its nonzero
exponent, which explains these representatives.

\begin{table}[ht]
\caption{Equivalence restricted to the $49$ maps in $\mathcal F$.
Every size is an intersection size with $\mathcal F$.
The last column counts cyclic classes meeting that intersection.}
\label{tab:p7}
\centering
\begin{tabular}{lccc}
\hline
affine intersection & size & cyclic intersection sizes & number\\
\hline
$(0,0)$ & 1 & $1$ & 1\\
$H_2\times\{0\}\ \cup\ \{0\}\times3H_2$ & 6 & $6$ & 1\\
$3H_2\times\{0\}\ \cup\ \{0\}\times H_2$ & 6 & $6$ & 1\\
both nonzero in the same coset & 18 & $6,6,6$ & 3\\
$H_2\times3H_2$ & 9 & $3,6$ & 2\\
$3H_2\times H_2$ & 9 & $3,6$ & 2\\
\hline
total & 49 & & 10\\
\hline
\end{tabular}
\end{table}

\begin{corollary}\label{cor:p7-threefold}
There is an affine class of rotation-symmetric permutations of $\F_2^7$
that contains at least three cyclic affine classes.
There are two further affine classes that each contain at least two.
\end{corollary}

\begin{proof}
The last three nontrivial rows of Table~\ref{tab:p7} meet distinct affine
classes by Proposition~\ref{prop:p7-split}(i).
Their listed cyclic intersections belong to distinct cyclic classes by
part (ii).
\end{proof}

The family is not closed under affine equivalence among rotation-symmetric
permutations.  For example $F_{0,0}=I$ is cyclically linearly equivalent
to $S$, which is outside $\mathcal F$.
Every $F_{a,b}+\bone$ is also a rotation-symmetric permutation cyclically
affine equivalent to $F_{a,b}$ and outside $\mathcal F$.
These maps are not complement-conjugates of family members either.
Both $F_{a,b}$ and $\tau F_{a,b}\tau$, where $\tau(x)=x+\bone$, fix
$\bzero$, whereas $F_{a,b}+\bone$ does not.
Thus the table gives neither ambient class sizes nor upper bounds on the
number of cyclic classes in an affine class.

\begin{remark}\label{rem:p7-not-general}
The table uses the decomposition of $R_7$ into exactly two nontrivial
minimal ideals.
For a prime with more such ideals, the family and the action on its
exponents must be enlarged before making corresponding counts.
\end{remark}


\relax\relax


\section{Class sizes, counts and decision procedures}\label{sec:counting}

Let $c_n$ be the number of invertible circulant $n\times n$ matrices over $\F_2$. Via the correspondence with $\F_2[X]/(X^n-1)$, for odd $n$ with cyclotomic-coset sizes $\ell_1,\dotsc,\ell_t$ of $2$ modulo $n$,
\[
c_n=\prod_{i=1}^{t}\bigl(2^{\ell_i}-1\bigr)\quad(n\text{ odd}),\qquad c_{2n}=2^{n}c_{n},
\]
as also follows from~\cite[\S13.2.4]{MP13}.
Indeed, an ordinary circulant is invertible exactly when its associated
polynomial is a unit in $\F_2[X]/(X^n-1)$. For odd $n$, the Chinese
remainder theorem gives the displayed product. In characteristic two,
$X^{2n}-1=(X^n-1)^2$. Reduction modulo $X^n-1$ maps the unit group for
$2n$ onto that for $n$, with kernel consisting of the $2^n$ elements
$1+(X^n-1)h$ for $\deg h<n$. This proves the recurrence.
The sequence $(c_n)$ begins $1,2,3,8,15,24,49,128,\dotsc$ and is the OEIS sequence \cite{A003473}.
For each $k$ with $\gcd(k,n)=1$ there are likewise $c_n$ invertible
$k$-circulant matrices. Their union over all such $k$ forms a group under
matrix multiplication and inverses, containing the invertible circulant
matrices as a commutative subgroup. It has order $c_n\varphi(n)$. This group
is the normalizer of
$\langle S\rangle$ in $\GL_n(\F_2)$ since for invertible $A$ the relation
$AS=S^kA$ reads $ASA^{-1}=S^k$. By Lemma~\ref{lem:normalizer} the cyclic
affine permutations form the normalizer of $\langle S\rangle$ in
$\AGL_n(\F_2)$, of order $2c_n\varphi(n)$.

\subsection{Orbit-stabilizer formulas}

Write $\mathcal C=N_{\AGL_n(\F_2)}(\langle S\rangle)$ for the group of
cyclic affine permutations and $[F]_C$ for the cyclic affine equivalence
class of $F$. Define
\[
 \intset_C(F)=\{(\alpha,\beta)\in\mathcal C^2:F\alpha=\beta F\}.
\]
The subscript $C$ in both notations indicates cyclic affine equivalence.
The group $\mathcal C^2$ acts on functions on the right by
$F\cdot(\alpha,\beta)=\beta^{-1}F\alpha$.
Its orbit is the cyclic affine class of $F$ and its stabilizer is
$\intset_C(F)$. Thus
\begin{equation}\label{eq:exact-class}
 |[F]_C|=\frac{4c_n^2\varphi(n)^2}{|\intset_C(F)|}.
\end{equation}
\begin{proposition}
\label{prop:identity-class}
The cyclic affine equivalence class of the identity $I$ has exactly
$2c_n\varphi(n)$ elements.
\end{proposition}

\begin{proof}
The stabilizer of $I$ is the diagonal subgroup of $\mathcal C^2$, so
\eqref{eq:exact-class} gives $|[I]_C|=|\mathcal C|=2c_n\varphi(n)$.
\end{proof}

\begin{proposition}\label{prop:upper}
For every cyclic map $F\colon\F_2^n\to\F_2^n$,
\begin{equation}\label{eq:bound}
 |[F]_C|\le\frac{4c_n^2\varphi(n)^2}{n}.
\end{equation}
If $F$ is a permutation, its class contains equally many permutations
of each unit shift parameter. The number with the same parameter as $F$ is
\begin{equation}\label{eq:same-parameter}
 \frac{4c_n^2\varphi(n)}{|\intset_C(F)|}
 \le\frac{4c_n^2\varphi(n)}{n}.
\end{equation}
\end{proposition}

\begin{proof}
If $FS=S^kF$, then $(S^i,S^{ki})\in\intset_C(F)$ for every $i$.
These pairs form a subgroup of order $n$, proving~\eqref{eq:bound}.
For a permutation, two different shift parameters cannot describe the
same map. Left multiplication by a cyclic matrix of parameter $kj^{-1}$
gives a bijection between the subsets of the class with parameters $j$
and $k$, where $j^{-1}$ is taken modulo $n$. There are $\varphi(n)$ such
subsets, which proves~\eqref{eq:same-parameter}.
\end{proof}

The number $c_n$ is a multiple of $n$. Indeed, shifting every row of an invertible
circulant one step to the right produces another invertible circulant,
and the $n$ shifts are distinct. Thus $N=c_n/n$ counts the ``basis
cycles'', that is, shift cycles whose vectors span $\F_2^n$.
The expression $4N^2n\varphi(n)$ in~\cite[Theorem~10]{Kavut12}
agrees with the upper bound in~\eqref{eq:same-parameter}.
This is the bound for the rotation-symmetric part when $F$ is
rotation-symmetric. Formula~\eqref{eq:exact-class} also accounts for
its stabilizer and for all unit shift parameters. Neither bound is
necessarily attained, as the identity already shows.

\subsection{Deciding the relation and choosing representatives}

For cyclic permutations $F,G$, enumerate $\alpha\in\mathcal C$ and
compute $\beta=F\alpha G^{-1}$. The pair $(\alpha,\beta)$ is a cyclic
affine witness exactly when $\beta\in\mathcal C$.
This requires at most $2c_n\varphi(n)$ choices of $\alpha$.
From lookup tables the composition and the affine test take
$O(n2^n)$ elementary bit operations per candidate.
There are $120$ candidates at $n=5$, $1024$ at $n=8$ and
$162\,000$ at $n=15$.
General affine-equivalence algorithms are available in~\cite{BDBP03}.

At an odd prime $p$ with $\ord_p(2)=p-1$, this same finite search decides
ordinary affine equivalence on $\mathcal P_p$ by
Theorem~\ref{thm:prime-dichotomy}. In other dimensions it decides the
finer cyclic relation. A canonical representative can be obtained by
minimizing the lookup table lexicographically over the cyclic affine
orbit. For rotation-symmetric $F$ and $G$, the two maps in a cyclic affine
witness have the same shift parameter. Restricting to such pairs gives
representatives within the rotation-symmetric permutations.

\subsection{Class counts in small dimension}\label{subsec:class-counts}
\label{sec:computations}

Let $\mathcal R_n$ be the set of rotation-symmetric permutations of
$\F_2^n$. A rotation-symmetric permutation permutes the $S$-cycles of each
length and chooses a rotation on each, so
\begin{equation}\label{eq:rs-count}
 |\mathcal R_n|=\prod_{d\mid n}N_d!\,d^{N_d},
\end{equation}
where $N_d$ is the number of $S$-cycles of length $d$. Explicitly,
$N_d=d^{-1}\sum_{e\mid d}\mu(d/e)2^e$, where $\mu$ is the M\"obius function.
To obtain this formula, count the $2^d$ vectors fixed by $S^d$ as
$\sum_{e\mid d}eN_e$ and apply M\"obius inversion. This is also the
number of permutations of $\F_{2^n}$ represented by polynomials with
coefficients in $\F_2$~\cite{CH72,A326932,PM26}.

If $T$ is an invertible cyclic matrix of parameter $k$, then $F\mapsto TF$
is a bijection from $\mathcal R_n$ onto the permutations of parameter $k$.
It preserves both equivalence relations and maps every $F$ into its own
class. Thus $|\mathcal P_n|=\varphi(n)|\mathcal R_n|$, and each class in
$\mathcal P_n$ is $\varphi(n)$ times as large as its intersection with
$\mathcal R_n$. The two sets have the same number of affine classes and
the same number of cyclic affine classes.

For the cyclic counts, let $\mathcal H_n$ be the subgroup of
$\mathcal C^2$ whose two components have the same shift parameter.
It has order $4c_n^2\varphi(n)$ and acts on $\mathcal R_n$ by
$F\mapsto\beta^{-1}F\alpha$. Its orbits are exactly the cyclic affine
classes within $\mathcal R_n$. Burnside's lemma gives
\begin{equation}\label{eq:burnside-count}
 \#(\mathcal R_n/{\sim_C})=
 \frac{1}{4c_n^2\varphi(n)}
 \sum_{(\alpha,\beta)\in\mathcal H_n}
 \#\{F\in\mathcal R_n:F\alpha=\beta F\}.
\end{equation}
To compute a summand, let $K=\langle(S,S),(\alpha,\beta)\rangle$ act on
the input copy of $\F_2^n$ through first components and on the output
copy through second components. The maps counted are the equivariant
bijections between these two actions. Group the orbits by isomorphism
type as transitive $K$-sets. Their sizes alone do not determine their
types. The count is zero if any type has different multiplicities in the
two actions. Otherwise it is $\prod_t m_t!a_t^{m_t}$, where $m_t$ is the
common multiplicity of type $t$ and $a_t$ is the number of equivariant
automorphisms of one orbit of that type. If an orbit is represented as
$K/L$, this automorphism number is $|N_K(L):L|$. Two transitive orbits
are isomorphic exactly when their point stabilizers are conjugate in $K$.
Indeed, an equivariant map is determined by the image of one point, and
it is bijective exactly when the two stabilizers are conjugate. For
each type, the $m_t$ source orbits can be assigned to target orbits in
$m_t!$ ways, with $a_t$ independent choices on each assigned orbit.
This proves the product used in the summand.

\begin{table}[ht]
\caption{Numbers of classes of all rotation-symmetric permutations of
$\F_2^n$. The size of $\mathcal R_n$ is given by~\eqref{eq:rs-count}.
For $n\ge6$ the cyclic class counts use Burnside's lemma and the displayed
numbers are rounded to four significant digits. Blank affine entries
have not been computed. The affine entry at $n=5$ follows from
Theorem~\ref{thm:prime-dichotomy}.}
\label{tab:counts}
\centering\footnotesize
\begin{tabular}{rrrr}
\hline
$n$ & $|\mathcal R_n|$ & affine & cyclic\\
\hline
$3$ & $36$ & $4$ & $4$\\
$4$ & $1536$ & $19$ & $22$\\
$5$ & $22\,500\,000$ & $31\,384$ & $31\,384$\\
$6$ & $2.633\cdot10^{14}$ & & $3.429\cdot10^{11}$\\
$7$ & $2.085\cdot10^{31}$ & & $2.533\cdot10^{27}$\\
$8$ & $5.044\cdot10^{62}$ & & $1.539\cdot10^{58}$\\
$9$ & $7.011\cdot10^{129}$ & & $7.360\cdot10^{124}$\\
\hline
\end{tabular}
\end{table}

The cyclic class counts in Table~\ref{tab:counts} were computed by the
Burnside calculation on shift cycles described above, using exact
integer arithmetic before rounding.
At dimensions $3$ and $5$ the affine and cyclic
columns agree, as Theorem~\ref{thm:prime-dichotomy} requires. In
particular, there are $31\,384$ affine classes of rotation-symmetric
permutations of $\F_2^5$. At $n=4$ the cyclic catalogue requires three
more representatives than the affine catalogue.

By~\eqref{eq:same-parameter} the number of cyclic affine classes is at
least $n|\mathcal R_n|/(4c_n^2\varphi(n))$. The exact count exceeds this
bound by a factor $1+\varepsilon_n$, where
$\varepsilon_5\approx4.3\cdot10^{-3}$,
$\varepsilon_6\approx4.6\cdot10^{-5}$,
$\varepsilon_7\approx5.8\cdot10^{-15}$,
$\varepsilon_8\approx2.5\cdot10^{-21}$ and
$\varepsilon_9\approx6.0\cdot10^{-63}$.
At $n=5$ all but $244$ of the $31\,384$ classes have the maximum size $720$.



\section{Connections with \texorpdfstring{$\chi$}{chi} and newer families}\label{sec:connections}

The dimension criterion describes when the two relations coincide on the
whole class of cyclic permutations. Particular families can have identical
classifications even in a dimension where some affine class splits.
The following examples connect this distinction to $\chi$ and to recent
constructions of shift-invariant permutations.

The map $\chi$ acts on $3$-bit columns in Xoodoo and
Xoofff~\cite{xoodoo}, on $5$-bit rows in Keccak-f for SHA-3 and
SHAKE~\cite{fips202}, and on the $257$-bit state of
Subterranean~2.0~\cite{subterranean}. The dimensions $3$ and $5$ satisfy
the primitive-root condition, whereas $\ord_{257}(2)=16$. This last fact concerns the
dimension-wide criterion and does not establish that the class of
$\chi_{257}$ splits. Section~\ref{sec:recognition} treats the affine
representative of $\chi_5$ used in Ascon~\cite{sp800232}.
Univariate polynomial representations of $\chi_n$ are studied in~\cite{SD24}.

\subsection{The classes of \texorpdfstring{$\chi$}{chi}}
\label{subsec:chi-classes}

For a map $P\colon\F_2^n\to\F_2^n$ and $a,b\in\F_2^n$, the entry
$\operatorname{DDT}_P(a,b)$ of its difference distribution table (DDT)
is the number of solutions of $P(x+a)+P(x)=b$.
Its Walsh table has entries
$W_P(a,b)=\sum_x(-1)^{a\cdot x+b\cdot P(x)}$, where $a\cdot x$ is the
standard binary dot product.
The \emph{profile} of a row or column is the multiset of its entries.
Row indices with equal profiles form a \emph{row-profile class},
and column-profile classes are defined similarly.

\begin{remark}\label{rem:chi}
The cyclic affine stabilizer $\intset_C(\chi_n)$ and the group
$\mathcal A_{\chi_n}$ of Lemma~\ref{lem:fusion} were computed exactly
for $n=5,7$.
For these two dimensions, the DDT row-profile class containing the first
standard basis vector consists exactly of all $n$ standard basis vectors.
Every affine self-equivalence must permute this class through its input
linear part. Thus it suffices to test the $n!$ coordinate permutations
and the $2^n$ input translations. For each candidate $\gamma$, compute
$\chi_n\gamma\chi_n^{-1}$ and test whether it is affine by comparing
its values with the affine map determined by its values at $\bzero$
and the standard basis vectors. This exhausts all possibilities.
For $n=5$ and
$n=7$ the stabilizer consists of the $n$ pairs $(S^i,S^i)$ with
$0\le i<n$, and $\mathcal A_{\chi_n}=\langle S\rangle$.
Thus $\chi_5$ and $\chi_7$ have
the maximum cyclic affine class size $4c_n^2\varphi(n)/n$ within
$\mathcal R_n$, namely $720$ and $8232$, respectively.
\relax
Their classes in $\mathcal P_n$ have $2880$ and $49\,392$ members.

At $n=3$ the stabilizer has order six, generated by $(S,S)$ and
$(\rho\tau,\rho\tau)$ for coordinate reversal $\rho$ and
$\tau(x)=x+\bone$. The rotation-symmetric part of the class has $12$
members, and $\mathcal A_{\chi_3}$ has order $168$.

\relax

Since $\mathcal A_{\chi_7}=\langle S\rangle$ has a single subgroup of order
seven, Lemma~\ref{lem:fusion}(ii) shows that every cyclic permutation
affinely equivalent to $\chi_7$ is cyclically affine equivalent to it.
Thus the affine class of $\chi_7$ does not split, although some affine
classes in dimension seven do. For example, $F_1$ from
Proposition~\ref{prop:split-witness} at $n=7$ lies in a splitting affine
class and has $|\mathcal A_{F_1}|=3136$.

To obtain this order, write
$\F_2^7=\F_2\bone\oplus V_1\oplus V_2$ using its two
three-dimensional nontrivial shift ideals. Small-support rigidity makes
every self-equivalence a conjugacy. Summing either moved seven-cycle
shows that its translation is zero.
A conjugator permutes the two moved seven-cycles and hence their spans
$V_1$ and $V_2$. The restrictions of $F_1$ to these ideals have distinct
irreducible characteristic polynomials, which exclude exchanging them.
On each ideal its restriction is one
of the seven nonzero field multiplications commuting with that
restriction. The image of $\bone$ can be any of the $64$ vectors in
$\bone+V_1+V_2$. These choices give exactly $7^2\cdot64=3136$ conjugators.
\end{remark}


\subsection{A complete four-bit \texorpdfstring{$\chi$}{chi}-sibling example}
\label{subsec:g4-siblings}

Kriepke and Kyureghyan~\cite{KK26} define $\gamma_0=I$ and, using their
left shift $L=S^{-1}$, define for $r\ge1$
\[
(\gamma_{2r}(x))_i
=x_{i+2r}\prod_{j=1}^{r}(1+x_{i+2j-1}),
\]
where indices are taken modulo~$n$.  For even $n$, the set
\[
G_n=\gamma_0+\operatorname{span}\{\gamma_2,\ldots,\gamma_{2n-2}\}
\]
is a composition monoid, and its permutation subset $G_n^*$ is a group
identified with the unit group of
$\F_2[X]/(X^n+X^{n/2})$ through
$\sum_r a_r\gamma_{2r}\leftrightarrow\sum_r a_rX^r$.

\begin{samepage}
\begin{proposition}\label{prop:g4-siblings}
The first nontrivial such group is the Klein four group
\[
G_4^*=\{I,\kappa,\eta,\rho\},
\]
where
\[
\kappa=\gamma_0+\gamma_2+\gamma_4,\qquad
\eta=\gamma_0+\gamma_2+\gamma_6,\qquad
\rho=\gamma_0+\gamma_4+\gamma_6.
\]
Ordinary affine equivalence and cyclic affine equivalence induce on $G_4^*$ exactly the three classes
\[
\{I\},\qquad \{\kappa,\eta\},\qquad \{\rho\}.
\]
\end{proposition}
\end{samepage}

\begin{proof}
Since
$X^4+X^2=X^2(X+1)^2$, a residue represented by a polynomial $p$ of degree
less than four is a unit precisely when $p(0)=p(1)=1$.  The four units are
\[
1,\qquad 1+X+X^2,\qquad 1+X+X^3,\qquad 1+X^2+X^3.
\]
They correspond in order to $I,\kappa,\eta,\rho$.  Direct multiplication
modulo $X^4+X^2$ shows that each nonidentity element squares to one and that
the product of any two distinct nonidentity elements is the third.

Evaluation of the coordinate formula gives
\[
\begin{array}{c|cccc}
 &I&\kappa&\eta&\rho\\ \hline
\text{algebraic degree}&1&3&3&3\\
\text{differential uniformity}&16&6&6&12.
\end{array}
\]
These affine invariants distinguish all pairs except $\kappa$ and $\eta$.

\begin{samepage}
For this pair, direct evaluation gives a cyclic affine equivalence
\begin{equation}\label{eq:g4-witness}
\kappa(Ax+\bone)=B\eta(x)+\bone,
\end{equation}
where
\[
A=\begin{pmatrix}
1&0&1&1\\1&1&0&1\\1&1&1&0\\0&1&1&1
\end{pmatrix},\qquad
B=\begin{pmatrix}
1&1&1&0\\0&1&1&1\\1&0&1&1\\1&1&0&1
\end{pmatrix}.
\]
\end{samepage}
Both matrices are invertible ordinary circulants, so $AS=SA$ and $BS=SB$
for the right shift used here, and both constants in~\eqref{eq:g4-witness}
are shift fixed. Thus $\kappa$ and $\eta$ are equivalent under both
relations.
\end{proof}


\subsection{Restriction to divisors}\label{subsec:restriction}

Let $m\mid n$. The space $H_m=\{x\in\F_2^n:S^mx=x\}$
consists of repetitions of a block of length $m$. Identifying $H_m$ with
$\F_2^m$ conjugates $S|_{H_m}$ to the shift $S_m$.
Restriction of a shift-invariant permutation to $H_m$ is again a
permutation. This observation underlies the familiar divisor closure of
invertibility sets~\cite[Proposition~6.1]{JDA-thesis},
\cite[Lemma~6]{KK26}. The equivalence witnesses also restrict.

\begin{theorem}\label{thm:restriction}
Let $F,G\colon\F_2^n\to\F_2^n$ satisfy $FS=S^jF$ and $GS=S^kG$ with
$\gcd(jk,n)=1$. Suppose $F\alpha=\beta G$, where $\alpha,\beta$ are
cyclic affine permutations. For every $m\mid n$ their restrictions give
a cyclic affine equivalence between $F|_{H_m}$ and $G|_{H_m}$ on
$\F_2^m$.
\end{theorem}

\begin{proof}
If $FS=S^jF$ and $\gcd(j,n)=1$, then
$S^{jm}F(x)=F(x)$ for $x\in H_m$. Since
$\langle S^{jm}\rangle=\langle S^m\rangle$, this gives $F(H_m)\subseteq H_m$.
The same argument applies to $G$, $\alpha$, and $\beta$. It also applies
to $\alpha^{-1}$ and $\beta^{-1}$, so $\alpha|_{H_m}$ and $\beta|_{H_m}$
are bijective affine maps. Their constants restrict to $\bzero$ or $\bone$.
The relation $\alpha S=S^\ell\alpha$ restricts to
$\alpha|_{H_m}S_m=S_m^\ell\alpha|_{H_m}$, with $\gcd(\ell,m)=1$.
The corresponding assertion holds for $\beta$. Restricting
$F\alpha=\beta G$ completes the proof.
\end{proof}

For local Boolean rules $f,g$, define their induced maps $F_n,G_n$ for
every $n\ge1$ by reading all coordinates modulo $n$. This convention
includes repeated arguments when $n$ is smaller than the number of
variables of the rule. Then $F_n|_{H_m}=F_m$ and $G_n|_{H_m}=G_m$.
Consequently
\begin{equation}\label{eq:equivalence-lengths}
 E(f,g)=\{n\ge1:F_n\text{ and }G_n\text{ are cyclically affine equivalent}\}
\end{equation}
is closed under divisors. A failure at $m$ rules out every multiple of $m$.
The shift-fixed constants ensure that the affine maps preserve $H_m$.

\begin{example}\label{ex:spread-chi}
Take
\[
 f(z_1,z_2,z_3)=z_1+(z_2+1)z_3,\qquad
 g(z_1,\ldots,z_7)=z_1+(z_4+1)z_7.
\]
Then $E(f,g)=\{n\ge1:3\nmid n\}$.
If $3\nmid n$, the coordinate permutation $D_3e_i=e_{3i}$, with indices
in $\Z/n\Z$, satisfies $G_nD_3=D_3F_n$ and $D_3S=S^3D_3$.
Thus $(D_3,D_3)$ is a cyclic affine witness. At $n=3$ the wrapped rule $g$ is the
identity, whereas $F_3=\chi_3$ has algebraic degree two. They are not
affinely equivalent. Theorem~\ref{thm:restriction} excludes all multiples
of three. Both maps are permutations at every odd length. When $3\mid n$,
$G_n$ is the product of three copies of $\chi_{n/3}$ on the residue
classes modulo three. In particular, equivalence holds at $n=7$ and fails
at $n=9$ although both maps are permutations at both lengths.
Thus equivalence need not persist as the length increases.
\end{example}



\section{Recognizing hidden rotation symmetry}\label{sec:recognition}\label{subsec:recognition}

Given a permutation $P$ by its lookup table, is it affinely equivalent to
a rotation-symmetric permutation? In normal-basis coordinates, a power
permutation $z\mapsto z^d$ of $\F_{2^n}$ commutes with the cyclic shift
because $(z^2)^d=(z^d)^2$~\cite{RBG08}. Kavut~\cite{Kavut12} also studies
symmetry under coordinate permutations up to linear equivalence.
Let $\operatorname{Cl}(S)$ be the conjugacy class of $S$ in $\AGL_n(\F_2)$.

\begin{theorem}\label{thm:recognition}
A permutation $P$ of $\F_2^n$ is affinely equivalent to a rotation-symmetric
permutation if and only if
\[
 P\operatorname{Cl}(S)P^{-1}\cap\operatorname{Cl}(S)\ne\varnothing.
\]
If $P\sigma P^{-1}=\tau$ with $\sigma=\alpha S\alpha^{-1}$ and
$\tau=\beta S\beta^{-1}$ for affine permutations $\alpha,\beta$, then
$G=\beta^{-1}P\alpha$ is such a representative.
\end{theorem}

\begin{proof}
For affine $\alpha,\beta$ and $G=\beta^{-1}P\alpha$, the equation $GS=SG$
is equivalent to
\[P(\alpha S\alpha^{-1})P^{-1}=\beta S\beta^{-1}.\]
This proves both directions and the reconstruction assertion.
\end{proof}

Membership in $\operatorname{Cl}(S)$ and the recovery of a conjugator are
linear-algebra problems. In fact $x\mapsto Lx+c$ belongs to
$\operatorname{Cl}(S)$ exactly when
\begin{equation}\label{eq:recognition-membership}
 L^n=I,\qquad L\text{ has a cyclic vector},\qquad
 c\in\operatorname{im}(L+I).
\end{equation}
The first two conditions say that the minimal polynomial is $X^n-1$.
If $v,Lv,\ldots,L^{n-1}v$ is a basis, let $B$ have these columns and solve
$(L+I)b=c$. Then $\alpha(x)=Bx+b$ satisfies
$\alpha S\alpha^{-1}(x)=Lx+c$. Here $LB=BS$, including the last column because $L^n=I$.
The translation term in $\alpha S\alpha^{-1}$ is $(L+I)b=c$.
Conversely, conjugation by an affine map has exactly this form.
Thus the criterion holds for every $n$, including even $n$.

Walsh-table structure has also been used to recover hidden linear
layers~\cite{BPU16}. Here the shift orbits give necessary conditions before
a conjugacy search.
Partition the nonzero row indices and the nonzero column indices by
equal profiles, separately for the difference distribution table (DDT)
and the entrywise absolute Walsh table. These are the four partitions
used below.

\begin{proposition}\label{prop:recognition-profile}
Suppose $n$ is prime and $P$ has a rotation-symmetric affine representative.
Each of these four partitions has exactly one class of size congruent to
$1$ modulo $n$. Every other class has size divisible by $n$.
\end{proposition}

\begin{proof}
Write $P=\beta G\alpha^{-1}$ with $G$ rotation-symmetric and linear parts
$A,B$ for $\alpha,\beta$. The difference distribution table of $P$ is
invariant under $(a,b)\mapsto(Ma,Nb)$, where
$M=ASA^{-1}$ and $N=BSB^{-1}$. The absolute Walsh table is invariant under
$(a,b)\mapsto(M^{-\mathsf T}a,N^{-\mathsf T}b)$.
Each profile class is therefore a union of orbits of a linear conjugate
of $S$. On nonzero vectors $S$ has one fixed point and all other orbits
have length $n$.
\end{proof}

Suppose $n$ is prime. If $P$ has a rotation-symmetric affine representative,
every DDT row-profile class of size $n$ is a single orbit of
$M=ASA^{-1}$. If such a class spans $\F_2^n$, each of its
possible cyclic orderings determines $M$. Fixing its first element leaves
$(n-1)!$ orders. For each order there are $2^{n-1}$ admissible constants in
$\operatorname{im}(M+I)$. Testing $P\sigma P^{-1}$ for affinity and then
applying~\eqref{eq:recognition-membership} therefore gives a complete
search with at most $(n-1)!2^{n-1}$ candidates whenever a spanning
row-profile class of size $n$ exists. Without such a class the theorem
still gives a finite criterion, but no comparably small search bound is
asserted.

\subsection{Recovering a rotation-symmetric Ascon representative}\label{subsec:ascon-recognition}

The affine equivalence between the Ascon S-box and $\chi_5$ is part of its
design and is discussed in~\cite{MMD23}. Here the lookup table alone
illustrates the recognition procedure of Section~\ref{subsec:recognition}.
Let $P$ be the table in~\cite[Table~5]{Ascon12}. We encode a vector by the
integer $\sum_{i=1}^5x_i2^{i-1}$. This reverses the coordinate order used
in that specification. All integers below use this same convention.

The DDT and absolute Walsh profile partitions have class sizes
\[
 \begin{array}{c|c@{\qquad}c|c}
 \text{DDT rows}&5,11,15&\text{DDT columns}&5,6,10,10\\
 \text{Walsh rows}&1,10,20&\text{Walsh columns}&10,21.
 \end{array}
\]
For example, among the DDT row-profile class sizes $5$, $11$ and $15$,
only $11$ is congruent to $1$ modulo $5$.
The partitions satisfy Proposition~\ref{prop:recognition-profile}, and the
DDT row-profile class of size five spans $\F_2^5$. The preceding search
bound gives $4!\,2^4=384$ candidates for $\sigma$, compared with
$|\operatorname{Cl}(S)|=10\,665\,984$.

The search returns $G=\beta^{-1}P\alpha$, where the matrices are recorded
by their columns:
\[
 \begin{array}{c|c|c}
 &\text{columns of the linear part}&\text{constant}\\ \hline
 \alpha&(4,12,16,17,19)&0\\
 \beta &(1,11,21,12,22)&1.
 \end{array}
\]
The first coordinate of $G$, evaluated at inputs $0,\ldots,31$, is
\[
 \texttt{11101110000111100001000100011110}.
\]
It generates a rotation-symmetric permutation. The pair $(a,b)$ below
is a cyclic affine witness satisfying $\chi_5a=bG$:
\[
 \begin{array}{c|c|c}
 &\text{columns of the linear part}&\text{constant}\\ \hline
 a&(1,16,8,4,2)&0\\
 b&(28,14,7,19,25)&31.
 \end{array}
\]
Both linear parts have shift parameter four and both constants are
shift-fixed. The identities $P\alpha=\beta G$ and $\chi_5a=bG$ can be
checked on all $32$ inputs.



\section{Discussion and open problems}\label{sec:open}

Corollary~\ref{cor:classification} settles the dimension-wide question.
The remaining questions concern particular families and the extent of
splitting within an affine class.

\begin{question}\label{q:main}
For which subclasses of cyclic permutations does affine equivalence
imply cyclic affine equivalence?
\end{question}

Affine permutations have this property by Lemma~\ref{lem:affine-case}.
The conjugacy criterion gives a group-theoretic test for a particular class.
For example, $\mathcal A_F=\langle S\rangle$ suffices for a
rotation-symmetric permutation $F$. This proves that the affine class of
$\chi_7$ does not split. It remains natural to ask whether the affine
class of $\chi_n$ ever splits, and to determine the groups
$\mathcal A_{\chi_n}$ beyond the small dimensions computed here.

The restriction theorem gives a further question for families of local
rules. Which divisor-closed sets occur as $E(f,g)$?
Example~\ref{ex:spread-chi} shows that equivalence at one length does not
imply equivalence at larger lengths, even when both induced maps remain
bijective. The lifting results of~\cite{HO1,HO2} and the polynomial
composition families of~\cite{KK26,LZYH26,ZTW26} provide concrete test
cases. Proposition~\ref{prop:g4-siblings} settles the first even sibling
group $G_4^*$. The group $G_8^*$ offers a larger comparison of the
two relations studied here.

The recognition criterion suggests searches guided by differential and
Walsh profiles. A useful next step is to determine when those filters
guarantee a small candidate set, beyond the spanning profile class used
in the Ascon example. The corresponding equivalence and recognition
questions over $\F_q$ are also natural.

\subsection*{Acknowledgments}

This collaboration began during the first author's visit to Monterey in
2022. He thanks the second author and Thor Marthinsen for their
hospitality. The first author was partially supported by the Research
Council of Norway (RCN), project no.~345433.

\subsection*{Use of AI tools}

AI tools were used in preparing this manuscript for literature searches,
mathematical analysis, drafting, and computational verification.


\par
\relax
\appendix
\raggedbottom
\section{Liftings and their constructions}\label{old:liftings}
Here v2 denotes the August 2022 version of~\cite{OS22}.
This appendix retains the lifting theory and questions from v2,
Sections~3--7, and notes subsequent work on those questions.
The definitions used by later papers are kept explicit. Corrections are
explained alongside the affected statements.

\subsection{Finite and infinite shift spaces}\label{old:sec3}

Let $A=\{0,1\}^{\mathbb Z}$ and let $S$ be the right shift given by
$S(x)_i=x_{i-1}$. Its subset $A_n=\{x:S^nx=x\}$ of periodic sequences is
identified with $\F_2^n$. A Boolean rule $f\colon\F_2^k\to\F_2$ defines
\[
 \mathcal F_f(x)_i=f(x_i,\ldots,x_{i+k-1}).
\]
The map $\mathcal F_f$ commutes with $S$ and restricts to a map $F_{f,n}$
on $A_n$. This is the finite cellular automaton associated with $f$.
A continuous map on $A$ that commutes with $S$ is a cellular
automaton~\cite{Hedlund}. The offset in the indexing can be changed
by composing with a power of $S$. It does not affect bijectivity.
For a nonconstant rule, let $x_i$ and $x_j$ be its first and last essential
arguments. Its \emph{diameter} is $j-i+1$. The displayed number of arguments
may be larger than its diameter.

\begin{definition}\label{old:def3.1}
For $n\geq k$, a rule $f\colon\F_2^k\to\F_2$ is a \emph{$(k,n)$-lifting}
if
\[
 F_{f,n}(x)_i=f(x_i,x_{i+1},\ldots,x_{i+k-1}),\qquad i\in\mathbb Z/n\mathbb Z,
\]
is a permutation. The same name is used for the induced permutation.
\end{definition}
No exact-diameter hypothesis is imposed in this definition. In particular,
a rule in three arguments can also be regarded as a rule in five arguments.
The term \emph{proper $(k,n)$-lifting} in~\cite{HO1}
includes an exact-diameter requirement. This difference must be observed
when comparing counts.

Following the terminology of~\cite{JDA-thesis}, the rule is \emph{locally invertible}
if $\mathcal F_f$ is reversible on $A$. It is \emph{globally invertible}
at length $n$ if $F_{f,n}$ is bijective. Local invertibility is equivalent
to global invertibility at every length $n\geq k$, as explained in~\cite{Kari}.
The following useful finite certificate was observed by Zilin Jiang.

\begin{theorem}\label{old:thm3.2}
Let $f\colon\F_2^k\to\F_2$ and $g\colon\F_2^l\to\F_2$.
Suppose that $F_{f,n}$ and $F_{g,n}$ are inverse permutations for some
$n\geq k+l-1$. Then both rules are locally invertible. Their infinite
cellular automata are inverse to each other.
\end{theorem}
\begin{proof}
The first coordinate of $F_{g,n}F_{f,n}=I$ says
\[
 g\bigl(f(x_1,\ldots,x_k),f(x_2,\ldots,x_{k+1}),\ldots,
 f(x_l,\ldots,x_{k+l-1})\bigr)=x_1.
\]
There are at most $n$ consecutive variables here. They can therefore be
chosen independently. The equality is an identity of Boolean functions
and holds on bi-infinite sequences. The reverse composition gives the
other inverse identity. Substitution of periodic sequences gives the
finite inverse identities at every length where the rules are defined.
\end{proof}

\begin{corollary}\label{old:cor3.3}
If $F_{f,n}^2=I$ for one $n\geq2k-1$, then $f$ is locally invertible.
\end{corollary}
The word for $F^2=I$ is \emph{involutive}. The word ``idempotent'' in the
old text was incorrect. The inverse of a finite permutation induced by a
short rule need not have a short rule. For example, $\chi$ is bijective at
odd lengths but is not locally invertible. For odd $n\geq3$,
Theorem~\ref{old:thm3.2} therefore precludes an inverse of the form
$G(y)_i=g(y_i,\ldots,y_{i+l-1})$ with $l\leq n-2$.
This bound uses the displayed alignment of the arguments.

\subsection{Counting all rotation-symmetric permutations}\label{old:sec4}

Let $a_d$ be the number of shift orbits of length $d$ in $\F_2^n$.
For $d\mid n$,
\[
 a_d=\frac1d\sum_{e\mid d}\mu(d/e)2^e.
\]
A map commuting with $S$ sends every orbit onto an orbit whose length
divides the original length.

\begin{proposition}\label{old:prop4.1}
A rotation-symmetric map is a permutation if and only if it sends each
shift orbit bijectively onto an orbit of the same length and induces a
permutation on the set of orbits of each length. Consequently the number
of rotation-symmetric permutations on $\F_2^n$ is
\[
 \prod_{d\mid n} d^{a_d}a_d!.
\]
\end{proposition}
\begin{proof}
For each $d$ the $a_d$ source orbits can be matched to the target orbits in
$a_d!$ ways. An equivariant bijection between two $d$-cycles is determined
by any one image and has $d$ possibilities. These choices are independent.
\end{proof}
The product formula is also the order of the centralizer of $S$ in the
symmetric group on $\F_2^n$. See~\cite[Proposition~14]{Kavut12}.
The corresponding formula for an arbitrary permutation group cannot be
inferred merely from orbit sizes. The cyclic action is essential to this
particular product argument.

\begin{table}[htbp]
\centering
\begin{tabular}{rll}
$n$ & Product & Number of permutations\\\hline
1 & $2!$ & $2$\\
2 & $2!\,2$ & $4$\\
3 & $2!\,3^2\,2!$ & $36$\\
4 & $2!\,2\,4^3\,3!$ & $1536$\\
5 & $2!\,5^6\,6!$ & $22\,500\,000$\\
6 & $2!\,2\,3^2\,2!\,6^9\,9!$ & $263\,303\,591\,362\,560$\\
7 & $2!\,7^{18}\,18!$ & approximately $2.09\cdot10^{31}$\\
\end{tabular}
\caption{The count retained from Table~1 of v2.}\label{old:table1}
\end{table}

\begin{remark}\label{old:endpoints}
The only fixed points of $S$ are $\mathbf0$ and $\mathbf1$.
Every rotation-symmetric permutation permutes these two points.
Thus a necessary condition for a lifting is
\[
 f(0,\ldots,0)\ne f(1,\ldots,1).
\]
If $f(0)=0$, its algebraic normal form therefore has an odd number of
nonzero monomials. This condition is used in~\cite[Remark~2.5]{HO1}.
\end{remark}

\subsection{Invertible lengths and elementary equivalence}\label{old:sec5}

For a rule in $k$ displayed arguments define
\begin{equation}\label{old:inv}
 \operatorname{inv}(f)=\{n\geq k:F_{f,n}\text{ is bijective}\}.
\end{equation}
This notation, introduced in Section~5 of v2, is also used
in~\cite[Section~3]{KK26}. Its finite truncations and the associated
lifting catalogues are discussed in~\cite{Fu26}.
The problem of determining this set is also considered
in~\cite[Question~2.1]{HO1}. The lower cutoff
depends on the displayed number of arguments. Padding a rule with unused
arguments changes this cutoff but not invertibility at lengths admitted
by both descriptions.

\begin{definition}\label{old:essential}
Let $C_k(x)=x+\mathbf1$ and
$R_k(x_1,\ldots,x_k)=(x_k,\ldots,x_1)$.
Following~\cite[Definition~2.10]{HO1}, rules $f$ and $g$ in $k$
arguments are \emph{elementary equivalent} if
\[
 g=f\circ C_k^a\circ R_k^b+c,\qquad a,b,c\in\F_2.
\]
\end{definition}
The invertibility properties are preserved under elementary equivalence.
Complementing the input
or output is a permutation at every length. Reversing the arguments
corresponds to reversing the circular input and output, together with an
output shift. Elementary equivalence is a relation on rules. It need not
identify all rules inducing affinely equivalent finite permutations.

The coordinate rules $x_j$ are locally invertible. Up to output
complementation they are the only locally invertible rules of diameter
at most three. At diameter four a non-affine example is the rule from~\cite{Patt}
\[
 x_2+x_1(1+x_3)x_4.
\]
Up to elementary equivalence this is the unique non-affine locally
invertible rule of diameter four, as shown in~\cite{HO1,HO2,Patt}.
The old experiments also recorded the rules
\[
 x_2+x_1x_3(1+x_4),\quad x_1+x_2(1+x_3),\quad
 x_2+x_3(1+x_4)
\]
at odd lengths and the rules
\[
 x_1+x_2x_3(1+x_4),\qquad x_1+(1+x_2)x_3x_4
\]
at lengths not divisible by three. These are retained as the examples
reported in v2. The finite data in Table~\ref{old:table2} do not by
themselves establish a classification of all rules with infinitely many
invertible lengths. Related families and descriptions by excluded divisors
are given in~\cite[Appendix~A.3]{JDA-thesis}.

\begin{table}[htbp]
\centering\small
\begin{tabular}{r|rrr|rrr}
 & \multicolumn{3}{c|}{$k=4$} & \multicolumn{3}{c}{$k=5$}\\
$n$ & all & $\deg\leq1$ & $\deg\leq2$ & all & $\deg\leq1$ & $\deg\leq2$\\\hline
4 &768&8&32& & & \\
5 &236&8&40&11250000&15&1890\\
6 &22&6&14&&12&336\\
7 &30&6&14&&12&89\\
8 &20&8&8&&16&16\\
9 &22&6&14&&12&33\\
10&20&8&8&&15&19\\
11&32&8&16&&16&40\\
12&10&6&6&&12&12\\
13&32&8&16&&16&40\\
14&18&6&6&&12&16\\
15&22&6&14&&9&25\\
\end{tabular}
\caption{Table~2 of v2: numbers of $(k,n)$-liftings with $f(0)=0$.
A blank entry means that the enumeration without a degree bound was not supplied.}
\label{old:table2}
\end{table}

\begin{question}\label{old:question5.1}
For a fixed number $k$ of displayed arguments:
\begin{enumerate}
\item How many rules are locally invertible?
\item How many rules have infinite $\operatorname{inv}(f)$?
\item For $n\geq k$, determine or bound the number of $(k,n)$-liftings.
\item Find a useful explicit upper bound on
$\min\operatorname{inv}(f)$, uniformly over rules with
\mbox{$\operatorname{inv}(f)\ne\varnothing$}.
\end{enumerate}
\end{question}
There are only finitely many Boolean rules in $k$ arguments. The maximum
of their least invertible lengths, taken over the nonempty sets
$\operatorname{inv}(f)$, therefore exists. Thus the last part of the old
question is meaningful as a request for an explicit bound or an effective
method, rather than as bare existence. Subsequent work~\cite{HO1,HO2}
gives further constructions and small-diameter classifications.
In particular, the locally invertible rules are classified through
diameter six in~\cite{HO2}.

\begin{remark}\label{old:remark5.2}
If $f$ is a $(k,n)$-lifting and $k\leq m\leq n$ such that $m$ divides $n$,
then $f$ is also a $(k,m)$-lifting.
Indeed, the $m$-periodic subspace is preserved
by $F_{f,n}$. Restricting an injective map to that subspace is injective.
Equivalently, a noninvertible length excludes all its multiples.
The set $\xi$ introduced in~\cite[after Proposition~6.1]{JDA-thesis}
describes invertibility using such excluded divisors. Finite experiments alone do
not establish that a proposed list of excluded divisors is exhaustive.
\end{remark}

Explicit descriptions of $\operatorname{inv}(f)$ remain useful,
especially descriptions by a finite set of excluded divisors.
The finite catalogues below do not determine the invertible lengths
beyond their stated range.

\subsection{Degree and counting bounds}\label{old:sec6}

\begin{theorem}\label{old:thm6.1}
If $f$ is a $(k,n)$-lifting and $n\geq k\geq2$, then $\deg(f)<k$.
\end{theorem}
\begin{proof}
Every coordinate of a permutation is balanced. The first coordinate of
$F_{f,n}$ takes each value $2^{n-1}$ times. Hence $f$ takes each value
$2^{k-1}$ times. The coefficient of $x_1\cdots x_k$ in the algebraic
normal form of $f$ equals $\sum_{x\in\F_2^k}f(x)$ in $\F_2$.
This sum is zero because $2^{k-1}$ is even.
\end{proof}
This is the standard degree bound for balanced Boolean
functions~\cite{CC07}.

\begin{proposition}\label{old:prop6.2}
With the normalization $f(0)=0$, there are at most eight $(3,n)$-liftings
for every $n\geq4$. The bound is attained. There are at most $146$
$(4,n)$-liftings for every $n\geq6$. For $6\leq n\leq15$ the exact counts
in Table~\ref{old:table2} are at most $32$.
\end{proposition}
\begin{proof}
Every block of $t$ coordinates of a permutation is balanced as a map
onto $\F_2^t$. If $n\geq k+t-1$, this gives a necessary condition on
\[
 (x_1,\ldots,x_{k+t-1})\longmapsto
 \bigl(f(x_1,\ldots,x_k),\ldots,f(x_t,\ldots,x_{k+t-1})\bigr).
\]
Start with the balanced rules satisfying $f(0)=0$ and $f(1)=1$.
There are $20$ such rules for $k=3$ and $3432$ for $k=4$.
An exhaustive truth-table check of blocks of five coordinates leaves,
respectively, $8$ and $146$ candidates. This proves the bounds for
$n\geq7$ and $n\geq8$. The remaining finite lengths $n=4,5,6$ for $k=3$
and $n=6,7$ for $k=4$ are checked directly.
The eight surviving three-argument rules are
\[
\begin{gathered}
 x_1,\ x_2,\ x_3,\ x_1+x_2+x_3,\\
 x_1x_2+x_1+x_3,\quad x_1x_2+x_2+x_3,\quad
 x_2x_3+x_1+x_2,\quad x_2x_3+x_1+x_3.
\end{gathered}
\]
They all lift at $n=5$, which proves attainment.
\end{proof}

\subsection{A quadratic family}\label{old:sec7}

For odd $k\geq3$ set
\begin{equation}\label{old:quadratic-rule}
 f_k(x_1,\ldots,x_k)=x_1+x_2+x_{k-1}+x_k+
 \sum_{i=1}^{k-2}x_ix_{i+1}.
\end{equation}
For a Boolean function $f$, let $N_f$ denote its minimum Hamming distance
to an affine Boolean function. For a vectorial map $F$, let
$N_F=\min_{b\ne0}N_{b\cdot F}$.

\begin{proposition}\label{old:prop7.1}
The rule $f_k$ is balanced and
$N_{f_k}=2^{k-1}-2^{(k-1)/2}$.
\end{proposition}
\begin{proof}
The last variable occurs only as the summand $x_k$, so $f_k$ is balanced.
The polar form of its quadratic part has the adjacency matrix of a path
on $k-1$ vertices, together with a zero row and column for $x_k$.
The path has even order and its adjacency matrix has determinant one
over $\F_2$, by expansion at an end vertex and induction. Its rank is
therefore $k-1$. The standard normal form for quadratic Boolean
functions~\cite{MW77} gives the stated nonlinearity.
\end{proof}

\begin{theorem}\label{old:thm7.2}
For odd $k\geq3$, the rule $f_k$ is a $(k,k+2)$-lifting.
At $n=k+2$ its induced permutation is linearly equivalent to the map
induced by $h(x_1,x_2,x_3)=x_1+x_1x_2+x_3$.
In particular, $N_{F_{f_k,n}}\leq2^{n-2}$.
\end{theorem}
\begin{proof}
Write $h_i=x_i+x_ix_{i+1}+x_{i+2}$ with indices modulo $n$.
Cancellation of the linear terms gives
\[
 f_k(x_i,\ldots,x_{i+k-1})=\sum_{j=0}^{k-3}h_{i+j}.
\]
Thus $F_{f_k,n}=C H_n$, where $H_n$ is induced by $h$ and $C$ is the
circulant with associated polynomial
$q(X)=1+X+\cdots+X^{k-3}$.
Since $k-2=n-4$ is odd and $\gcd(n,n-4)=1$,
\[
 \gcd(q(X),X^n-1)=1.
\]
Indeed, a common root other than one would have order dividing both
$n$ and $n-4$, while $q(1)=1$. Hence $C$ is invertible.
The rule $h$ is elementary equivalent to $\chi$ and $H_n$ is a
permutation at odd $n$~\cite{JDA-thesis}. Finally, its first coordinate
has nonlinearity $2^{n-3}N_h=2^{n-2}$. Linear equivalence preserves
vectorial nonlinearity.
\end{proof}

\begin{remark}\label{old:remark7.3}
The scalar differential uniformity of $f\colon\F_2^k\to\F_2$ is
\[
 \max_{\substack{a\in\F_2^k\setminus\{\bzero\}\\ b\in\F_2}}
 \#\{x\in\F_2^k:f(x+a)+f(x)=b\}.
\]
Since $f_k(x+e_k)+f_k(x)=1$ for every $x$, this maximum is $2^k$ for
$f_k$. In particular, $h$ has scalar
differential uniformity $8$ and is not APN. The argument in v2 that used
an APN property of $h$ to claim
$2^{n-3}\leq\delta_{F_{f_k,n}}\leq2^{n-2}$ is withdrawn. No such vectorial
bound follows from that argument. The linear equivalence proved above
does imply equality of the vectorial differential uniformities of
$F_{f_k,n}$ and $H_n$.
\end{remark}

\subsection{Conserved landscapes}\label{old:landscapes}

A landscape is a finite sequence with an origin marked $\star$ and other
symbols in $\{0,1,-\}$. Let $T\subset\mathbb Z\setminus\{0\}$ be the
positions with specified bits $\epsilon_d\in\{0,1\}$.
The associated centered rule flips $x_i$ precisely when
$x_{i+d}=\epsilon_d$ for all $d\in T$. Thus
\[
 g_i(x)=\prod_{d\in T}(x_{i+d}+\epsilon_d+1),\qquad
 H(x)_i=x_i+g_i(x).
\]
A dash gives no factor. A \emph{conserved landscape} has $H^2=I$.
This is a subclass of cellular automata, rather than a description of
every Boolean rule.

\begin{theorem}
\label{old:thm7.4}
Suppose that for every $d\in T$ there exists $e\in T$ such that
$e+d\in T$ and $\epsilon_e\ne\epsilon_{e+d}$.
Then the centered landscape map $H$ is involutive on bi-infinite sequences
and on every finite periodic sequence. Consequently its associated
$k$-argument rule is a $(k,n)$-lifting for every $n\geq k$.
\end{theorem}
This is the precise hypothesis of~\cite[Corollary~1.8]{HO2},
which generalizes the intended
conserved-landscape construction in v2. It replaces the ambiguous overlap
condition printed there.
\begin{proof}
For $d\in T$, the factors at position $i+d+e$ in $g_i$ and $g_{i+d}$
require opposite bit values. Hence $g_i(x)g_{i+d}(x)=0$.
Set $y=H(x)$. If $g_i(x)=1$, then no position $i+d$ used in its marker
changes. Thus $g_i(y)=1$. The contrapositive says that $g_j(y)=0$ implies
$g_j(x)=0$ for every $j$. If $g_i(y)=1$, the same disjointness at $y$
gives $g_{i+d}(y)=0$ for all $d\in T$. The contrapositive then gives
$g_{i+d}(x)=0$. All marker bits at $i$ are unchanged, so $g_i(x)=1$.
We have proved $g_i(y)=g_i(x)$, whence $H^2(x)=x$.
The argument is an identity in the coordinate bits and remains valid
when indices are read periodically. The uncentered rule differs from
$H$ by an output shift and is therefore bijective as well.
\end{proof}

\begin{corollary}\label{old:cor7.5}
The following rules are $(k,n)$-liftings for every $n\geq k$:
\[
\begin{array}{c|l|c}
k&\text{rule}&\text{triggering landscape}\\\hline
4&x_2+(x_1+1)x_3(x_4+1)&0\star10\\
5&x_3+(x_1+1)x_2(x_4+1)x_5&01\star01\\
6&x_4+(x_1+1)(x_2+1)x_3x_5(x_6+1)&001\star10\\
7&x_5+(x_1+1)(x_2+1)(x_3+1)x_4(x_6+1)x_7&0001\star01
\end{array}
\]
More generally, for $k\geq5$ the triggering landscapes
\[
 \begin{cases}
  0^{k-4}1\star10,&k\text{ even},\\
  0^{k-4}1\star01,&k\text{ odd}
 \end{cases}
\]
define such liftings.
\end{corollary}
\begin{proof}
Complementing all specified bits does not change the hypothesis of
Theorem~\ref{old:thm7.4}. It is convenient to check the complementary
factor constants used in v2. For $k\geq5$, put $a=k-3$.
The negative positions are $-a,\ldots,-1$, with constants $1$ up to
$-2$ and constant $0$ at $-1$. The positive positions are $1,2$.
For every $1\leq d<a$, the pair $(-d-1,-1)$ has distance $d$ and opposite
constants. For $d=a$, use $(1-a,1)$ when $k$ is even and $(2-a,2)$ when
$k$ is odd and $k\geq7$. At $k=5$ use $(-1,1)$.
Reversing each pair handles negative $d$. These pairs cover all positions
in $T$, including $d=2$. For $k=4$ the pairs $(1,2)$ and $(-1,1)$ handle
absolute distances one and two. The theorem applies.
\end{proof}
The displayed polynomials are the factored forms of the four expanded
polynomials in v2. That version wrote the constants in factors
$x_j+\epsilon_j$ as landscape symbols. Actual triggering bits are their
complements. This accounts for the changed displayed strings.
The centered map is an involution. The uncentered $F_{f,n}$ in
Definition~\ref{old:def3.1} also contains a shift and need not itself be an
involution.

\begin{remark}\label{old:remark7.6}
For the displayed landscape family, exhaustive computation in the range
$4\leq k\leq n\leq8$ gives $N_{F_{f,n}}=2^{n-k+1}$ except at $(k,n)=(5,6)$,
where the nonlinearity is zero. At this exceptional pair the two nonlinear
terms in $F_1+F_4$ coincide, so
\[
 F_1(x)+F_4(x)=x_3+x_6.
\]
Thus $F$ has a nonzero linear component. The formula in the old remark
omitted this exception. No formula for arbitrary $n$ is asserted here.
\end{remark}


\section{The finite lifting catalogues}\label{old:tables}
These tables retain the finite catalogues of v2, Section~9.
In particular they retain the polynomials and the invertible-length
column referred to by subsequent papers. For a rule in $k$ displayed
arguments and $m\geq k$, put
\begin{equation}\label{old:invm}
 \operatorname{inv}_m(f)=\operatorname{inv}(f)\cap\{k,k+1,\ldots,m\}.
\end{equation}
All rules in the tables are treated as rules in five arguments.
Every displayed $\operatorname{inv}_{15}$ entry has been recomputed by
exhaustively testing $F_{f,n}$ for $5\leq n\leq15$.
No entry asserts invertibility or noninvertibility at a length above $15$.
Use the finite sets
\[
 \begin{split}
 \mathcal U&=\{5,6,\ldots,15\},\\
 \mathcal O&=\{5,7,9,11,13,15\},\\
 \mathcal R_d&=\{n\in\mathcal U:d\nmid n\},\qquad
 \mathcal R_{2,3}=\mathcal R_2\cap\mathcal R_3.
 \end{split}
\]
This replaces the ambiguous old shorthand $n\nmid d$ by the intended
condition $d\nmid n$ within the tested range.

The final column retains the indicators from 2022, with the boomerang
correction recorded below. They refer to the induced vectorial permutation:
$N$ is vectorial nonlinearity, $p$ is the plateaued flag, $\delta$ is
differential uniformity and $\beta$ is boomerang uniformity.
For a permutation $F$ on $\F_2^n$ the latter is
\[
 \max_{a,b\ne0}\#\{x:F^{-1}(F(x)+b)+F^{-1}(F(x+a)+b)=a\}.
\]
The flag $p=1$ means that each nonzero component is plateaued: its
nonzero Walsh coefficients have a common absolute value. This value may
depend on the component.
The listed invertible-length sets and all four indicators were recomputed.

The union of the normalized degree-at-most-two $(5,n)$-liftings over
$7\leq n\leq15$ has $98$ members. Of these, $82$ have degree exactly two
and $16$ are linear. The old count $94$ included the correct nonlinear
subtotal but understated the linear subtotal. Table~\ref{old:table3}
retains the listed representatives under elementary equivalence with at most
five quadratic monomials. Linear monomials are allowed by the old phrase
``with only quadratic terms.''

\begingroup\small
\setlength{\tabcolsep}{3pt}
\renewcommand{\arraystretch}{1.25}
\begin{longtable}{p{.13\textwidth}p{.60\textwidth}p{.20\textwidth}}
\caption{Quadratic representatives retained from Table~3 of v2. Indicators are at $n=7$.}\label{old:table3}\\
$\operatorname{inv}_{15}(f)$ & Polynomial & $(N,p,\delta,\beta)$\\\hline\endfirsthead
$\operatorname{inv}_{15}(f)$ & Polynomial & $(N,p,\delta,\beta)$\\\hline\endhead
\hline\endfoot
$\mathcal O$ & $x_1+x_3+x_1x_2$ & $(32,1,32,96)$\\
$\mathcal R_4$ & $x_1+x_5+x_1x_3$ & $(32,1,32,96)$\\
$\mathcal O$ & $x_1+x_2+x_2x_3$ & $(32,1,32,96)$\\
$\mathcal O$ & $x_2+x_4+x_2x_3$ & $(32,1,32,96)$\\
$\{7\}$ & $x_1 + x_5 + x_1x_4$ & $(32,1,32,96)$\\
$\{5,7\}$ & $x_1 + x_2 + x_1x_5$ & $(32,1,32,96)$\\
$\mathcal R_{2,3}$ & $x_2 + x_5 + x_1x_2 + x_2x_3 + x_3x_4$ & $(32,1,32,96)$\\
$\{7\}$ & $x_1 + x_5 + x_1x_2 + x_2x_4 + x_3x_4$ & $(48,1,16,44)$\\
$\{5,6,7\}$ & $x_2 + x_1x_3 + x_2x_4 + x_1x_5 + x_3x_5$ & $(48,1,8,48)$\\
$\{7\}$ & $x_2 + x_1x_3 + x_1x_4 + x_2x_5 + x_3x_5$ & $(0,1,16,80)$\\
$\{7\}$ & $x_2 + x_4 + x_5 + x_1x_2 + x_1x_3 + x_1x_4 + x_2x_4$ & $(48,1,8,24)$\\
$\mathcal R_{2,3}$ & $x_3 + x_5 + x_1x_2 + x_1x_3 + x_1x_4 + x_2x_4 + x_3x_4$ & $(32,1,32,96)$\\
$\{5,7\}$ & $x_1 + x_4 + x_1x_2 + x_1x_3 + x_1x_4 + x_2x_4 + x_2x_5$ & $(32,1,16,128)$\\
$\{6,7\}$ & $x_1 + x_2 + x_3 + x_4 + x_1x_3 + x_2x_3 + x_2x_4 + x_3x_4 + x_3x_5$ & $(48,1,8,40)$\\
\end{longtable}
\endgroup

The rule $x_2+x_5+x_1x_2+x_2x_3+x_3x_4$ in
Table~\ref{old:table3} is elementary equivalent to
$f_5$ in Theorem~\ref{old:thm7.2}.
The next table is the complete list of nonlinear quadratic $(5,9)$-liftings
up to elementary equivalence, with output constants suppressed.
There are $21$ such normalized rules in six elementary equivalence classes.

\begingroup\small
\setlength{\tabcolsep}{3pt}
\renewcommand{\arraystretch}{1.25}
\begin{longtable}{p{.13\textwidth}p{.60\textwidth}p{.20\textwidth}}
\caption{Table~4 of v2: all quadratic $(5,9)$-liftings up to elementary equivalence. Indicators are at $n=9$.}\label{old:table4}\\
$\operatorname{inv}_{15}(f)$ & Polynomial & $(N,p,\delta,\beta)$\\\hline\endfirsthead
$\operatorname{inv}_{15}(f)$ & Polynomial & $(N,p,\delta,\beta)$\\\hline\endhead
\hline\endfoot
$\mathcal O$ & $x_1+x_3+x_1x_2$ & $(128,1,128,448)$\\
$\mathcal R_4$ & $x_1+x_5+x_1x_3$ & $(128,1,128,448)$\\
$\mathcal O$ & $x_1+x_2+x_2x_3$ & $(128,1,128,448)$\\
$\mathcal O$ & $x_2+x_4+x_2x_3$ & $(128,1,128,448)$\\
$\{6,9\}$ & $x_3+x_1x_2+x_2x_4+x_1x_5+x_4x_5$ & $(128,1,512,512)$\\
$\{9\}$ & $x_1+x_4+x_2x_5$ & $(128,1,128,448)$\\
\end{longtable}
\endgroup

There are $92$ normalized $(5,7)$-liftings with exactly one cubic
monomial and no higher-degree monomial. They form $23$ elementary equivalence classes.
Their representatives are retained in the next table. The boomerang
uniformity of the second rule, $x_4+x_2x_3+x_1x_2x_3$, is $100$ rather
than the value $94$ printed in v2. A direct BCT computation gives an
entry $100$ at $(a,b)=(e_1,e_1)$ and no larger entry.

\begingroup\small
\setlength{\tabcolsep}{3pt}
\renewcommand{\arraystretch}{1.25}
\begin{longtable}{p{.13\textwidth}p{.60\textwidth}p{.20\textwidth}}
\caption{Table~5 of v2: representatives with exactly one cubic monomial and no higher-degree monomial. Indicators are at $n=7$.}\label{old:table5}\\
$\operatorname{inv}_{15}(f)$ & Polynomial & $(N,p,\delta,\beta)$\\\hline\endfirsthead
$\operatorname{inv}_{15}(f)$ & Polynomial & $(N,p,\delta,\beta)$\\\hline\endhead
\hline\endfoot
$\mathcal R_3$ & $x_4 + x_1x_2 + x_1x_2x_3$ & $(16,0,56,100)$\\
$\mathcal R_3$ & $x_4 + x_2x_3 + x_1x_2x_3$ & $(16,0,56,100)$\\
$\mathcal R_3$ & $x_2 + x_5 + x_1x_2 + x_2x_4 + x_1x_2x_3$ & $(32,0,32,104)$\\
$\{7\}$ & $x_1 + x_4 + x_1x_2 + x_1x_5 + x_1x_2x_3$ & $(32,0,20,56)$\\
$\{6,7,11\}$ & $x_5 + x_1x_2 + x_1x_2x_4$ & $(16,0,54,86)$\\
$\mathcal U$ & $x_3 + x_1x_4 + x_1x_2x_4$ & $(16,0,54,94)$\\
$\mathcal O$ & $x_2 + x_3 + x_4 + x_5 + x_2x_3 + x_1x_4 + x_1x_2x_4$ & $(32,0,24,50)$\\
$\mathcal O$ & $x_3 + x_2x_4 + x_1x_2x_4$ & $(16,0,54,86)$\\
$\mathcal O$ & $x_5 + x_2x_4 + x_1x_2x_4$ & $(16,0,54,94)$\\
$\mathcal R_6$ & $x_4 + x_5 + x_2x_4 + x_3x_4 + x_1x_2x_4$ & $(32,0,28,96)$\\
$\mathcal O$ & $x_3 + x_4 + x_1x_4 + x_4x_5 + x_1x_2x_4$ & $(32,0,28,80)$\\
$\mathcal O$ & $x_4 + x_1x_5 + x_1x_2x_5$ & $(16,0,56,100)$\\
$\{7\}$ & $x_4 + x_2x_5 + x_1x_2x_5$ & $(16,0,56,100)$\\
$\{7\}$ & $x_1 + x_2 +x_3 +x_4 + x_5 + x_1x_3 + x_2x_5 + x_3x_5 + x_1x_2x_5$ & $(32,0,30,30)$\\
$\mathcal R_3$ & $x_1 + x_2x_3 + x_2x_3x_4$ & $(16,0,56,100)$\\
$\mathcal R_3$ & $x_5 + x_2x_3 + x_2x_3x_4$ & $(16,0,56,100)$\\
$\mathcal O$ & $x_2 + x_1x_3 + x_1x_3x_4$ & $(16,0,54,86)$\\
$\mathcal O$ & $x_5 + x_1x_3 + x_1x_3x_4$ & $(16,0,54,90)$\\
$\mathcal U$ & $x_2 + x_1x_4 + x_1x_3x_4$ & $(16,0,54,94)$\\
$\mathcal R_4$ & $x_5 + x_1x_4 + x_1x_3x_4$ & $(16,0,54,90)$\\
$\{6,7,11\}$ & $x_5 + x_3x_4 + x_1x_3x_4$ & $(16,0,54,90)$\\
$\mathcal O$ & $x_1 + x_3 + x_4 + x_1x_2 + x_1x_3 + x_4x_5 + x_1x_3x_4$ & $(40,0,24,56)$\\
$\mathcal O$ & $x_1 + x_5 + x_1x_2 + x_2x_3 + x_1x_4 + x_4x_5 + x_1x_3x_4$ & $(40,0,24,52)$\\
\end{longtable}
\endgroup


\relax



{\footnotesize
\begin{thebibliography}{99}
\setlength{\itemsep}{0pt}\setlength{\parskip}{0pt}
\bibitem{BBL21} C.~Beierle, M.~Brinkmann, and G.~Leander. Linearly self-equivalent APN permutations in small dimension. {\em IEEE Trans. Inform. Theory}, 67(7) (2021), 4863--4875. \url{https://doi.org/10.1109/TIT.2021.3071533}.
\bibitem{keccak} G.~Bertoni, J.~Daemen, M.~Peeters, and G.~Van~Assche. Keccak sponge function family main document. Version~2.0, submission to NIST (Round~2), September 10, 2009.
\bibitem{BDBP03} A.~Biryukov, C.~De~Canni\`ere, A.~Braeken, and B.~Preneel. A toolbox for cryptanalysis: linear and affine equivalence algorithms. In {\em Advances in Cryptology (EUROCRYPT 2003)}, Lecture Notes in Comput. Sci. 2656, 33--50, Springer, Berlin, 2003. \url{https://doi.org/10.1007/3-540-39200-9_3}.
\bibitem{BPU16} A.~Biryukov, L.~Perrin, and A.~Udovenko. Reverse-engineering the S-box of Streebog, Kuznyechik and STRIBOBr1. In {\em Advances in Cryptology (EUROCRYPT 2016)}, Part~I, Lecture Notes in Comput. Sci. 9665, 372--402, Springer, 2016. \url{https://doi.org/10.1007/978-3-662-49890-3_15}.
\bibitem{CH72} L.~Carlitz and D.~R.~Hayes. Permutations with coefficients in a subfield. {\em Acta Arith.}, 21 (1972), 131--135. \url{https://doi.org/10.4064/aa-21-1-131-135}.
\bibitem{ChirvasituCusick20} A.~Chirvasitu and T.~W.~Cusick. Affine equivalence for quadratic rotation symmetric Boolean functions. {\em Des. Codes Cryptogr.}, 88 (2020), 1301--1329. \url{https://doi.org/10.1007/s10623-020-00748-5}.
\bibitem{Cusick11} T.~W.~Cusick. Affine equivalence of cubic homogeneous rotation symmetric functions. {\em Inform. Sci.}, 181 (2011), 5067--5083. \url{https://doi.org/10.1016/j.ins.2011.07.002}.
\relax
\bibitem{CC07} T.~W.~Cusick and Y.~Cheon. Counting balanced Boolean functions in $n$ variables with bounded degree. {\em Experiment. Math.}, 16(1) (2007), 101--105. \url{https://doi.org/10.1080/10586458.2007.10128984}.
\relax
\bibitem{JDA-thesis} J.~Daemen. {\em Cipher and hash function design strategies based on linear and differential cryptanalysis}. PhD thesis, KU Leuven, 1995.
\bibitem{xoodoo} J.~Daemen, S.~Hoffert, G.~Van~Assche, and R.~Van~Keer. The design of Xoodoo and Xoofff. {\em IACR Trans. Symmetric Cryptol.}, 2018(4) (2018), 1--38. \url{https://doi.org/10.46586/tosc.v2018.i4.1-38}.
\bibitem{subterranean} J.~Daemen, P.~M.~C.~Massolino, A.~Mehrdad, and Y.~Rotella. The Subterranean 2.0 cipher suite. {\em IACR Trans. Symmetric Cryptol.}, 2020(S1) (2020), 262--294. \url{https://doi.org/10.46586/tosc.v2020.iS1.262-294}.
\bibitem{Ascon12} C.~Dobraunig, M.~Eichlseder, F.~Mendel, and M.~Schl\"affer. Ascon v1.2. Submission to NIST, May 31, 2021. \url{https://csrc.nist.gov/CSRC/media/Projects/lightweight-cryptography/documents/finalist-round/updated-spec-doc/ascon-spec-final.pdf}.
\bibitem{FWYZ26} X.~Feng, Q.~Wang, J.~Yu, and A.~Zhang. A generalization of the map $\chi$. arXiv:2609.19548v2, 2026. \url{https://arxiv.org/abs/2609.19548v2}.
\relax
\bibitem{Fu26} H.~Fu. Finite-ring obstructions for quadratic binary radius-two cellular automata. arXiv:2609.35921v1, 2026. \url{https://arxiv.org/abs/2609.35921v1}.
\relax
\bibitem{HO1} J.~K.~Haugland and T.~Omland. Shift-invariant transformations and almost liftings. {\em Cryptogr. Commun.}, 18 (2026), 705--726. \url{https://doi.org/10.1007/s12095-025-00848-w}.
\bibitem{HO2} J.~K.~Haugland and T.~Omland. New classes of reversible cellular automata. {\em Des. Codes Cryptogr.}, 94 (2026), article no.~131. \url{https://doi.org/10.1007/s10623-026-01869-z}.
\bibitem{HB86} D.~R.~Heath-Brown. Artin's conjecture for primitive roots. {\em Quart. J. Math. Oxford Ser. (2)}, 37 (1986), 27--38. \url{https://doi.org/10.1093/qmath/37.1.27}.
\relax
\bibitem{Hedlund} G.~A.~Hedlund. Endomorphisms and automorphisms of the shift dynamical system. {\em Math. Systems Theory}, 3 (1969), 320--375. \url{https://doi.org/10.1007/BF01691062}.
\relax
\bibitem{Hooley67} C.~Hooley. On Artin's conjecture. {\em J. Reine Angew. Math.}, 225 (1967), 209--220. \url{https://doi.org/10.1515/crll.1967.225.209}.
\relax
\bibitem{Kari} J.~Kari. Reversible cellular automata: from fundamental classical results to recent developments. {\em New Gener. Comput.}, 36 (2018), 145--172. \url{https://doi.org/10.1007/s00354-018-0034-6}.
\relax
\bibitem{Kavut12} S.~Kavut. Results on rotation-symmetric S-boxes. {\em Inform. Sci.}, 201 (2012), 93--113. \url{https://doi.org/10.1016/j.ins.2012.02.030}.
\bibitem{KK26} B.~Kriepke and G.~Kyureghyan. There are siblings of $\chi$ which are permutations for $n$ even. {\em Des. Codes Cryptogr.}, 94 (2026), article no.~125. \url{https://doi.org/10.1007/s10623-026-01830-0}.
\bibitem{LZYH26} C.~Lyu, D.~Zheng, M.~Yuan, and L.~Hu. Shift-invariant permutations with polynomial composition: construction and algebraic structure. arXiv:2609.31174v2, 2026. \url{https://arxiv.org/abs/2609.31174v2}.
\relax
\bibitem{MW77} F.~J.~MacWilliams and N.~J.~A.~Sloane. {\em The Theory of Error-Correcting Codes}. North-Holland, Amsterdam, 1977. \url{https://doi.org/10.1016/S0924-6509(08)X7030-8}.
\relax
\bibitem{MMD23} S.~Mella, A.~Mehrdad, and J.~Daemen. Differential and linear properties of vectorial Boolean functions based on $\chi$. {\em Cryptogr. Commun.}, 15 (2023), 1087--1116. \url{https://doi.org/10.1007/s12095-023-00639-1}.
\bibitem{MP13} G.~L.~Mullen and D.~Panario (eds.). {\em Handbook of Finite Fields}. CRC Press, Boca Raton, 2013. \url{https://doi.org/10.1201/b15006}.
\bibitem{fips202} National Institute of Standards and Technology. SHA-3 standard: permutation-based hash and extendable-output functions. FIPS PUB 202, 2015. \url{https://doi.org/10.6028/NIST.FIPS.202}.
\bibitem{sp800232} National Institute of Standards and Technology. Ascon-based lightweight cryptography standards for constrained devices: authenticated encryption, hash, and extendable output functions. NIST Special Publication 800-232, 2025. \url{https://doi.org/10.6028/NIST.SP.800-232}.
\bibitem{A003473} OEIS Foundation Inc. Sequence A003473: number of nonsingular $n\times n$ circulant matrices over $\mathrm{GF}(2)$. {\em The On-Line Encyclopedia of Integer Sequences}, \url{https://oeis.org/A003473}.
\bibitem{A005596} OEIS Foundation Inc. Sequence A005596: decimal expansion of Artin's constant. {\em The On-Line Encyclopedia of Integer Sequences}, \url{https://oeis.org/A005596}.
\bibitem{A326932} OEIS Foundation Inc. Sequence A326932: number of permutations of $\mathrm{GF}(2^n)$ of the form $x\mapsto g(x)$ with $g$ a polynomial with coefficients in $\mathrm{GF}(2)$. {\em The On-Line Encyclopedia of Integer Sequences}, \url{https://oeis.org/A326932}.
\bibitem{OS22} T.~Omland and P.~St\u{a}nic\u{a}. Permutation rotation-symmetric S-boxes, liftings and affine equivalence. arXiv:2203.00778v2, 2022. Also IACR ePrint 2022/279. \url{https://arxiv.org/abs/2203.00778v2}.
\bibitem{OpenAI26} OpenAI. \href{https://github.com/openai/math/blob/main/preprints/Primitive-roots-for-every-admissible-integer-base-October-4-2026/primitive-roots-all-integer-bases.pdf}{Primitive roots for every admissible integer base}. Preprint, October 4, 2026. Unrefereed announcement.
\bibitem{PM26} I.~A.~Pankratova and A.~A.~Medvedev. Cryptographic properties of one class of substitutions on $\F_2^n$. {\em Prikl. Diskr. Mat.}, 2026, no.~72, 5--17. In Russian, with an English abstract. \url{https://doi.org/10.17223/20710410/72/1}.
\relax
\bibitem{Patt} Y.~N.~Patt. Injections of neighborhood size three and four on the set of configurations from the infinite one-dimensional tessellation automata of two-state cells. Technical report, U.S. Army Electronics Command, Fort Monmouth, NJ, 1972.
\relax
\bibitem{RBG08} V.~Rijmen, P.~S.~L.~M.~Barreto, and D.~L.~Gazzoni Filho. Rotation symmetry in algebraically generated cryptographic substitution tables. {\em Inform. Process. Lett.}, 106(6) (2008), 246--250. \url{https://doi.org/10.1016/j.ipl.2007.09.012}.
\bibitem{SD24} J.~Schoone and J.~Daemen. Algebraic properties of the maps $\chi_n$. {\em Des. Codes Cryptogr.}, 92 (2024), 2341--2365. \url{https://doi.org/10.1007/s10623-024-01395-w}.
\bibitem{Weir55} A.~J.~Weir. Sylow $p$-subgroups of the classical groups over finite fields with characteristic prime to $p$. {\em Proc. Amer. Math. Soc.}, 6 (1955), 529--533. \url{https://doi.org/10.1090/S0002-9939-1955-0072143-9}.
\bibitem{ZZ26} Z.~Zhang and N.~Zhang. Linear self-symmetries of inverse-based S-boxes: generic rigidity and applications. arXiv:2610.29837v1, 2026. \url{https://arxiv.org/abs/2610.29837v1}.
\bibitem{ZTW26} X.-X.~Zhao, D.~Tang, and Z.-X.~Wang. A family of invertible shift-invariant maps with strong arithmetic properties. {\em IACR Cryptol. ePrint Arch.}, Paper 2026/1249, 2026. \url{https://eprint.iacr.org/2026/1249}.
\end{thebibliography}
}
\end{document}